\documentclass{amsart}
\usepackage{amssymb}
\usepackage{stmaryrd} 
\usepackage{amsmath} 
\usepackage{amscd}
\usepackage{amsthm}
\usepackage{amsbsy}
\usepackage[margin=1.2 5in]{geometry}
\usepackage{commath, longtable}
\usepackage{comment, enumerate}
\usepackage{geometry}
\usepackage[matrix,arrow]{xy}
\usepackage{hyperref}
\usepackage{mathrsfs}
\usepackage{color}
\usepackage{float}
\usepackage{mathtools,caption}
\usepackage[table,dvipsnames]{xcolor}
\usepackage{tikz-cd}
\usepackage{longtable}
\usepackage[utf8]{inputenc}
\usepackage[OT2,T1]{fontenc}
\usepackage[normalem]{ulem}
\usepackage{hyperref}
\usepackage[customcolors]{hf-tikz}
\usepackage{enumitem}
\newlist{primenumerate}{enumerate}{1}
\setlist[primenumerate,1]{label={\arabic*$'$}}
\hypersetup{
 colorlinks=true,
 linkcolor=DarkOrchid,
 filecolor=blue,
 citecolor=olive,
 urlcolor=orange,
 pdftitle={Gambheera-Kundu project},
 }
\usepackage{booktabs}

\DeclareSymbolFont{cyrletters}{OT2}{wncyr}{m}{n}
\DeclareMathSymbol{\Sha}{\mathalpha}{cyrletters}{"58}
\DeclareMathSymbol{\Zha}{\mathalpha}{cyrletters}{"11}

\newcommand{\DK}[1]{\textcolor{purple}{#1}}

\newcommand{\edit}[1]{\textcolor{black}{#1}}

\newtheorem{theorem}{Theorem}[section]
\newtheorem{lemma}[theorem]{Lemma}

\newtheorem{proposition}[theorem]{Proposition}
\newtheorem{corollary}[theorem]{Corollary}
\newtheorem{definition}[theorem]{Definition}

\numberwithin{equation}{section}
\newtheorem{lthm}{Theorem}

\theoremstyle{remark}
\newtheorem{remark}[theorem]{Remark}

\newcommand{\val}{\vec{\boldsymbol{\alpha}}}
\newcommand{\veb}{\vec{\boldsymbol{\beta}}}
\newcommand{\rk}{\operatorname{rk}}
\newcommand{\tr}{\operatorname{tr}}
\newcommand{\EC}{\mathsf{E}}
\newcommand{\fp}{\mathfrak{p}}
\newcommand{\fm}{\mathfrak{m}}
\newcommand{\cG}{\mathcal{G}}
\newcommand{\cM}{\mathcal{M}}
\newcommand{\cL}{\mathcal{L}}

\newcommand{\Gal}{\operatorname{Gal}}
\newcommand{\Qp}{\mathbb{Q}_p}
\newcommand{\sG}{\mathsf{G}}
\newcommand{\Kn}{\mathbb{K}_{n}}
\newcommand{\K}{\mathbb{K}}
\newcommand{\Zp}{\mathbb{Z}_p}

\newcommand{\Z}{\mathbb{Z}}
\newcommand{\A}{\mathbb{A}}
\newcommand{\Q}{\mathbb{Q}}
\newcommand{\cyc}{\operatorname{{cyc}}}
\newcommand{\ord}{\operatorname{ord}}
\newcommand{\image}{\operatorname{im}}
\newcommand{\Sel}{\operatorname{Sel}}
\newcommand{\charac}{\operatorname{char}}
\newcommand{\cok}{\operatorname{coker}}
\renewcommand{\exp}{\operatorname{exp}}

\makeatletter
\theoremstyle{plain} %makes the name of the conjecture/theorem bold
\newtheorem*{intr@thm}{\intr@thmname}

\newtheorem*{c@njecture}{\conjn@name}
\newcommand{\myl@bel}[2]{
 \protected@write \@auxout {}{\string \newlabel {#1}{{#2}{\thepage}{#2}{#1}{}} }
 \hypertarget{#1}{}
 } %showing an error here

 \renewcommand{\pod}[1]{\allowbreak\mathchoice
  {\if@display \mkern 18mu\else \mkern 8mu\fi (#1)}
  {\if@display \mkern 18mu\else \mkern 8mu\fi (#1)}
  {\mkern4mu(#1)}
  {\mkern4mu(#1)}
}
\makeatother

\makeatletter
\newcommand{\mylabel}[2]{#2\def\@currentlabel{#2}\label{#1}}
\makeatother

\title[]{Structure of (Fine) Mordell--Weil Groups}

\author[R.~Gambheera]{Rusiru Gambheera}
\address[Gambheera]{Department of Mathematics, University of California Santa Barbara, CA 93106, USA}
\email{rusiru@ucsb.edu}

\author[D.~Kundu]{Debanjana Kundu}
\address[Kundu]{Department of Mathematics and Statistics\\ University of Regina \\ Sasktachewan\\ Canada}
\email{debanjana.kundu@uregina.ca}

\begin{document}

\begin{abstract}
In this article we study the algebraic structure of fine Mordell--Weil groups, plus/minus Mordell--Weil groups, Selmer groups, and plus/minus Selmer groups in the cyclotomic $\Zp$-extensions of abelian number fields.
As a first, we prove theorems on the equivariant structure of fine Mordell--Weil groups and plus/minus Mordell--Weil groups.
In other words, we study the explicit shape of the fine, plus/minus objects as a $\Lambda(\cG)$-module with $\cG \simeq \Zp \times \sG$ and $\sG$ a finite abelian group.
We prove refinements of previously known results over $\Q$ for the classical Selmer group and the plus/minus Selmer group, and subsequently also the Shafarevich--Tate group, and the plus/minus Shafarevich--Tate group.
This gives new evidence towards an affirmative answer for the Kurihara--Pollack problem.
\end{abstract}



\bigskip
\bigskip
\maketitle

\section{Introduction}

In the study of rational points on \edit{a}belian varieties, the Selmer group plays an important role.
In analogy with the work of K.~Iwasawa on growth of class groups in $\Zp$-extensions, B.~Mazur introduced the study of Selmer groups of elliptic curves in such extensions in \cite{Maz72}, by exploiting the intimate connection between class groups and Selmer groups.
More recently, in \cite{CS05}, J.~Coates--R.~Sujatha initiated the systematic study of a subgroup of the classical Selmer group called the \emph{fine Selmer group}.
They showed that the fine Selmer group has stronger finiteness properties than the classical Selmer group.
In fact, it `interpolates' between the classical Selmer group and the class group.

In this article, we study refinements of the structure of (fine) Mordell--Weil groups over the cyclotomic $\Zp$-extensions of $\Q$ and more general number fields (denoted by $K$).
Throughout we consider elliptic curves $\EC/\Q$ with good reduction at an odd prime $p$.
In particular, we provide evidence towards an affirmative answer for (generalizations of) problems posed by R.~Greenberg and M.~Kurihara--R.~Pollack.

\subsection*{Main Results}
We will now highlight some of the main results that we prove in this paper.

\subsubsection*{Structure of fine Mordell--Weil groups}

A key advantage of studying the fine Selmer group is that it approximates the ideal class group better than the classical Selmer group \textit{and} the results can often be stated uniformly independent of the reduction type.
In particular, analogous to the classical Iwasawa $\mu=0$ conjecture it is expected (see \cite[Conjecture~A]{CS05}) that the dual fine Selmer group \edit{over the cyclotomic $\Z_p$-extension $K_{\cyc}/K$, denoted by} $\Sel_0(\EC/K_{\cyc})^\vee$, is a finitely generated $\Zp$-module.
The fine Mordell--Weil group introduced by C.~Wuthrich \cite{Wut07} plays a key role in studying the fine Selmer group.
A primary goal in this paper is to understand the algebraic structure of fine Mordell--Weil groups in more detail.
By \cite[Theorem~3.3]{Lei23}, we know that the $\Lambda$-module structure of the fine Mordell--Weil group is
\[
\cM(\EC/K_{\cyc})^\vee \sim \bigoplus_{i=1}^{u} \frac{\Lambda}{\langle \Phi_{c_i}(x) \rangle} \text{ for some integers } c_i, u,
\]
\edit{where $\sim$ denotes pseudo-isomorphism of $\Lambda$-modules.}
Here, $\Phi_k(x)$ denotes the $p^k$-cyclotomic polynomial.
On the other hand, our goal is to count the \emph{number of times} each $\Phi_k(x)$ arises in the structure of $\cM(\EC/K_{\cyc})^\vee$.

In the following result, we have technical notation whose explicit description will be introduced in Definition~\ref{theta n}.
Intuitively, $e_n$ is an indicator of whether there is rank jump from the $(n-1)$-th layer to the $n$-th layer of the cyclotomic $\Zp$-extension of $K$.
\edit{For $n\ge0$, write $K_{(n)}$ for the unique sub-extension of $K_{\cyc}/K$ of degree $p^n$ and define 
\[
e_n=
\begin{cases}
\frac{\rk_{\Z}\EC(K_{(n)})-\rk_{\Z}\EC(K_{(n-1)})}{\phi(p^n)} & \text{ when } n>0,\\
\rk_{\Z}\EC(K) & \text{ when } n=0,
\end{cases}
\]
where  $\phi$ denotes the Euler totient function.}
Whereas, $\theta_n$ measures the growth of $p$-adic image, due to the rank growth at the $n$-th layer.

\begin{lthm}
\label{thm A}
Let $K/\Q$ be an abelian extension with Galois group $\sG$ such that $K \cap \Q_{\cyc}= \Q$.
Let \edit{the exponent of the group $\sG$ denoted by $\exp(\sG)$} be $m$ and $p\geq 3$ be a fixed odd prime which does not split in $\Q(\zeta_m)/\Q$.
Let $\EC/\Q$ be an elliptic curve with good reduction at $p$.
Suppose that the $p$-primary fine Shafarevich--Tate group $\Zha(\EC/K_{(n)})[p^\infty]$ is finite for each $n$.
Then the $\Lambda$-structure of the dual fine Mordell--Weil group is given by
\[
\cM(\EC/K_{\cyc})^\vee \sim \bigoplus_{n\geq 0} \left(\frac{\Lambda}{\langle\Phi_n(x)\rangle}\right)^{e_n - \theta_n}.
\]
In particular, the characteristic ideal of the dual fine Mordell--Weil group is given by
\[
\charac_{\Lambda}(\cM(\EC/K_{\cyc})^\vee) = \Bigg\langle\prod_{n\geq 0} (\Phi_n(x))^{e_n -\theta_n} \Bigg\rangle.
\]
\end{lthm}

Even without the hypothesis that $\edit{\exp}(\mathsf{G})=m$ and the fixed odd prime $p$ is non-split in $\Q(\zeta_m)$, we can recover at least one divisibility; see Remark~\ref{rem: one div}.
Subsequently, we also study an equivariant refinement of this above result in Theorem~\ref{thm 4.12}, i.e., we provide a precise structure of $\cM(\EC/K_{\cyc})^\vee$ as a $\Lambda[\sG]$-module.
We discuss the situation involving non-abelian extensions briefly in Remark~\ref{Rem non abelian}.

\subsubsection*{Generalization of a problem of Greenberg}
%For an integer $n\ge0$, we .
In \cite[Problem~0.7]{KP07}, Greenberg posed the following problem 
\begin{equation}
\label{Gr}\tag{Gr}
\charac_\Lambda(\Sel_0(\EC/\Q_{\cyc})^\vee)\stackrel{?}{=}\left(\prod_{e_n>0}\Phi_n(x)^{e_n-1}\right).
\end{equation}
In particular, it suggests that $\Sel_0(\EC/\Q_{\cyc})^\vee$ could be pseudo-isomorphic to $\Sel_0(\EC/\Q_{\cyc})^{\iota,\vee}$ as $\Lambda$-modules, where $\iota$ denotes the inversion sending a group-like element $\sigma\in\Lambda$ to $\sigma^{-1}$, and $M^\iota$ denotes the $\Lambda$-module obtained from a $\Lambda$-module $M$, where the action is twisted by $\iota$.

If $\Zha(\EC/\Q_{\cyc})[p^\infty]$ is finite (as predicted by Wuthrich in \cite[Conjecture~8.2]{Wut07}) %and $\Sha(\EC/\Q_{(n)})$ is finite at each layer of $\Q_{\cyc}/\Q$, 
then Greenberg's question is known to have a positive answer; see \cite[Theorem~C]{Lei23}.
Note that in the statement of \cite[Theorem~C]{Lei23} it is assumed that $\Sha(\EC/\Q_{(n)})[p^\infty]$ is finite for all $n$ but the proof only requires that $\Zha(\EC/\Q_{(n)})[p^\infty]$ is finite for each layer.

\subsubsection*{Generalization of a problem of Kurihara--Pollack}
When $\EC/\Q$ has good \emph{supersingular} reduction at $p$ with $a_p(\EC) = 0$, Kurihara and Pollack posed the following problem in \cite[Problem~3.2]{KP07}
\begin{equation}
\label{KP}\tag{KP}
\gcd(L_p^+, L_p^-)\stackrel{?}{=}\left(x^{e_0}\prod_{\substack{n>0\\ e_n>0}}\Phi_n(x)^{e_n-1}\right).
\end{equation}
where $L_p^{\pm}$ are Pollack’s $p$-adic $L$-functions defined in \cite{Pol03}.
We remind the reader that the relationship between the problem posed by Greenberg and Kurihara--Pollack ha\edit{s} been studied in \cite[Section~3]{KP07}.
In order to provide an answer to this question, A.~Lei introduced the notion of $\pm$-Mordell--Weil groups and proved that the description of $\gcd(L_p^+, L_p^-)$ predicted by the Kurihara--Pollack problem is equal to the greatest common divisor of the characteristic ideals of the Pontryagin duals of the plus and minus Mordell--Weil groups.
In particular, he proved in \cite[Theorem~D]{Lei23} that if $\Sha(\EC/\Q_{(n)})[p^\infty]$ is finite for all $n$ then
\[
\gcd(\charac_{\Lambda}(\cM^+(\EC/\Q_{\cyc})^\vee), \charac_{\Lambda}(\cM^-(\EC/\Q_{\cyc})^\vee) )= \Bigg\langle x^{e_0}\prod_{\substack{n>0\\ e_n>0}}\Phi_n(x)^{e_n-1} \Bigg\rangle.
\]
Furthermore, he explains that the answer to the Kurihara--Pollack problem is affirmative under the hypotheses that certain $\pm$ Shafarevich--Tate groups are finite and that Kobayashi’s plus and minus main conjecture holds.
For the definition of these objects see Section~\ref{prelim: good supersingular}.
These results were generalized in \cite{Ray23} to the case of signed Selmer groups and signed Shafarevich--Tate groups of abelian varieties.

Our contribution is that we prove a version of Lei's theorem when the base field is a more general number field $K$.
We emphasize that the main result we prove also provides a refinement of \cite[Theorem~D]{Lei23} when $K=\Q$.
In particular, we provide explicit description of the characteristic ideal of the $\pm$ Mordell--Weil group and not just the structure of the $\gcd$ of the characteristic ideals of the $\pm$ Mordell--Weil groups.
As per the knowledge of the authors, even in the case when $K=\Q$ previous results focused on the structure of the $\gcd$ alone.
More precisely, we prove the following result.

\begin{lthm}
\label{thm B}
Let $\EC/\Q$ be an elliptic curve and $p$ be a fixed prime of good \emph{supersingular} reduction.
Let $K/\Q$ be an abelian extension with Galois group $\sG$ such that $p$ is unramified in $K$.
Let $\edit{\exp}(\sG) = m$ such that $p$ does not split in $\Q(\zeta_m)/\Q$.
Suppose that $a_p(\EC)=0$.
Let $\cM^{\pm}(\EC/K_{\cyc})$ denote the $\pm$ Mordell--Weil groups over $K_{\cyc}$.
Further assume that $\Zha(\EC/K_{(n)})[p^\infty]$, $\Zha^{\pm}(\EC/K_{(n)})[p^\infty]$ are finite for all $n$.
Then
\begin{align*}
    \charac_{\Lambda}(\cM^{+}(\EC/K_{\cyc})^\vee) & =\Bigg\langle\ x^{e_0} \prod_{\substack{n>0\\ n \text{ odd}}} \Phi_n(x)^{e_n - \theta_n} \prod_{\substack{n>0\\ n \text{ even}}} \Phi_n(x)^{e_n}\Bigg\rangle\\\
     \charac_{\Lambda}(\cM^{-}(\EC/K_{\cyc})^\vee) & = \Bigg\langle x^{e_0} \prod_{\substack{n>0\\ n \text{ odd}}} \Phi_n(x)^{e_n} \prod_{\substack{n>0\\ n \text{ even}}} \Phi_n(x)^{e_n  - \theta_n} \Bigg\rangle.
\end{align*}
In particular,
\[
\gcd\left( \charac_{\Lambda}(\cM^{+}(\EC/K_{\cyc})^\vee), \charac_{\Lambda}(\cM^{-}(\EC/K_{\cyc})^\vee)\right) =\Bigg\langle x^{e_0} \prod_{n>0} \Phi_n(x)^{e_n - \theta_n} \Bigg\rangle.
\]
\end{lthm}

Using our techniques, one can also prove versions of Theorem~\ref{thm A} and Theorem~\ref{thm B} for certain non-abelian extensions; see Remark~\ref{Rem non abelian}.
In addition, the same results can be obtained for the anti-cyclotomic Iwasawa tower of imaginary quadratic fields under some conditions; see Remark~\ref{anticyclotomic remark}.

\subsubsection*{Refinements of the structure of Selmer groups}
The third theme of this article involves studying the structure of the \textit{dual Selmer group} over the extension $\Q_{\cyc}/\Q$.
We study this separately in the good ordinary and the good supersingular reduction case.
Under reasonable hypothesis on the finiteness of the Shafarevich--Tate group and when $p$ is a prime of good ordinary reduction, J.~Lee \edit{\cite{Lee20}} provided precise structure of the dual Shafarevich--Tate group when
\[
\Sel(\EC/\Q_{\cyc})^\vee \sim \left( \bigoplus_{i=1}^r \frac{\Lambda}{\langle g_i^{d_i}(x)\rangle}\right) \oplus \left(\bigoplus_{\substack{j=1\\ f_j\geq 2}}^s \frac{\Lambda}{\langle \Phi_{a_j}^{f_j}(x)\rangle}\right) \oplus \left(\bigoplus_{k=1}^t \frac{\Lambda}{\langle \Phi_{b_k}(x)\rangle} \right),
\]
where the $g_i$'s are coprime to the cyclotomic polynomials.
In particular, he showed that the summand involving $\Phi_{b_k}(x)$ does not appear in the structure of the dual Shafarevich--Tate group.
We prove a version of this result in the case when $p$ is a prime of good \emph{supersingular} reduction.
Using the work of T.~Kataoka \edit{\cite{Kat21}} and other results proven in this paper, we provide finer results on the structure of the dual Selmer group (resp.~$\pm$ dual Selmer group) and the dual Shafarevich--Tate group ($\pm$ dual Shafarevich--Tate group) over the   extension $\Q_{\cyc}/\Q$ when Shafarevich--Tate group (resp.~$\pm$ Shafarevich--Tate group) is finite at each layer of the cyclotomic $\Zp$-extension and
\[
\Sel^{\pm}(\EC/\Q_{\cyc})^\vee \sim \left( \bigoplus_{i=1}^r \frac{\Lambda}{\langle g_i^{d_i}(x)\rangle}\right) \oplus \left(\bigoplus_{\substack{j=1\\ f_j\geq 2}}^s \frac{\Lambda}{\langle \Phi_{a_j}^{f_j}(x)\rangle}\right) \oplus \left(\bigoplus_{k=1}^t \frac{\Lambda}{\langle \Phi_{b_k}(x)\rangle} \right).
\]
In particular, we prove the following result.

\begin{lthm}
Fix an elliptic curve $\EC/\Q$ and an odd prime $p$ of good \emph{ordinary} (resp.~\emph{supersingular}) reduction.
Suppose that $\Sha(\EC/\Q_{(n)})[p^{\infty}]$ \textup{(}resp.~$\Zha^{\pm}(\EC/\Q_{(n)})[p^{\infty}]$\textup{)} is finite for all $n$ and that $\Zha(\EC/\Q_{\cyc})[p^\infty]$ is finite.
Then $g_i(x)$'s and $a_j$'s are distinct.

Also, $\Sha(\EC/\Q_{\cyc})[p^\infty]^\vee$ \textup{(}resp.~$\Zha^{\pm}(\EC/\Q_{\cyc})[p^\infty]^\vee$\textup{)} is pseudo-isomorphic to a cyclic $\Lambda$-module.
%Also, the cyclotomic-part of  $\Sha(\EC/\Q_{\cyc})[p^\infty]^\vee$ \textup{(}resp.~$\Zha^{\pm}(\EC/\Q_{\cyc})[p^\infty]^\vee$\textup{)} is cyclic as a $\Lambda$-module.
\end{lthm}

The upshot is that knowing how the Mordell--Weil rank grows in $\Q_{\cyc}/\Q$ along with information on the $\Lambda$-characteristic ideal of $\Sel(\EC/\Q_{\cyc})$ -- which can be obtained from an $L$-function by the Iwasawa Main Conjecture -- completely determines the structure of $\Sel(\EC/\Q_{\cyc})^\vee$.

%As a corollary we can also prove the following result which gives stronger evidence for the Kurihara--Pollack problem.
%As compared to \cite[Theorem~D]{Lei23} we need milder hypotheses.
As a corollary we can also prove the following result which gives new evidence towards an affirmative answer for the Kurihara--Pollack problem.

\begin{lthm}\label{Theorem D}
Fix an elliptic curve $\EC/\Q$ and an odd prime $p$ such that $a_p =0$.
Suppose that $\Zha^{\pm}(\EC/\Q_{(n)})[p^{\infty}]$ is finite for all $n$ and that $\Zha(\EC/\Q_{\cyc})[p^\infty]$ is finite.
Then for $t\geq e_0$
\[
\text{cyclotomic part of } \gcd\left( \charac_{\Lambda}(\Sel^{+}(\EC/\Q_{\cyc})^\vee), \charac_{\Lambda}(\Sel^{-}(\EC/\Q_{\cyc})^\vee)\right) = \Bigg\langle x^{t} \prod_{\substack{n\geq 1\\e_n>1}} \Phi^{e_n - 1}_n(x) \Bigg\rangle.
\]
If $\Zha^{\pm}(\EC/\Q_{\cyc})[p^\infty]^{\Gamma}$ is finite, then $t=e_0$.
\end{lthm}

One interesting implication of Theorem~\ref{Theorem D} is that if we assume Kobayashi's plus or minus main conjecture and the finiteness of $\Sha(\EC/\Q_{(n)})[p^{\infty}]$ at each level $n$, then (\ref{Gr}) implies (\ref{KP}).
This can be seen by the main result of \cite{Lei-Sujatha-21}.
Moreover, notice that our conditions are different from \cite[Proposition~3.4(2)]{KP07} which is also a theorem of similar flavour.
The conditions of our theorem are also different from those of \cite[Remark~6.10]{Lei23} in the following sense.
We work directly with the structure of cyclotomic part of $\Sel^{\pm}(\EC/\Q_{\cyc})$ without going through $\cM^{\pm}(\EC/K_{\cyc})$.
Therefore, we do not have to assume the finiteness of $\Zha^{\pm}(\EC/\Q_{\cyc})[p^\infty]$.

\subsection*{Organization}
Including this introduction, this article has six section.
Section~\ref{sec: preliminary} is preliminary in nature where we discuss the basic notions and definitions which are used throughout this paper.
Section~\ref{sec: algebraic preliminary} is a section on algebraic preliminaries where we prove technical (algebraic) results that are useful in the remainder of the paper.
In Section~\ref{sec: Rank Growth Data and Fine Selmer Groups} we first derive arithmetic implications of the technical results proven in the previous section.
We then prove results on the $\Lambda$-structure of fine Selmer groups and also prove equivariant results.
Section~\ref{sec: pm MW group} focuses on the case when $p$ is a prime of supersingular reduction.
In particular, we study a generalization of the Kurihara--Pollack question upon base change and provide refinement of solutions towards the original problem (over $\Q$).
In Section~\ref{sec: Selmer group}, we use our earlier results to study the structure of the Selmer group over the cyclotomic $\Zp$-extension when the base field is $\Q$.

\subsection*{Outlook}
A possible future direction is to remove the hypotheses that we are forced to impose in the theorems that we prove.
Further understanding of $\Zha(\EC/K_{\cyc})[p^\infty]$ and $\Zha^{\pm}(\EC/K_{\cyc})[p^\infty]$ will be instrumental in formulating a generalized version of the problem posed by Greenberg (resp.~Kurihara--Pollack) for number fields.

\subsection*{Data Availability} There is no associated data with this manuscript.

\section*{Acknowledgements}
The authors acknowledge the support of their respective AMS--Simons early career travel grants.
\edit{During the revision stages of this article, DK was supported by the NSERC Discovery Grant RGPIN-2026-07384.}
We thank Francesc Castella, Antonio Lei, and Katharina M{\"u}ller for helpful discussions.
We thank Ben Forrás for his comments on an earlier version of the draft.
\edit{We are grateful to the referee for their careful and timely reading of our manuscript.}

\section{Notation and Basic Definitions}
\label{sec: preliminary}

\subsection{Iwasawa theory}

Fix a prime $p$.
We write $\Q_{\cyc}$ to denote the cyclotomic $\Zp$-extension of $\Q$.
This is the maximal totally real pro-$p$ subfield of $\bigcup_n \Q(\zeta_{p^n})$ where $\zeta_{p^n}$ denotes a primitive $p^n$-th root of unity.
In particular, there is a tower
\[
\Q = \Q_{(0)} \subset \Q_{(1)} \subset \ldots \subset \Q_{\cyc}
\]
such that each $\Q_{(n)}$ is the totally real subfield of $\Q(\zeta_{p^{n+1}})$ with $\Gal(\Q_{(n)}/\Q)\simeq \Z/p^n\Z$.
By infinite Galois theory,
\[
\Gamma = \Gal(\Q_{\cyc}/\Q) \simeq \varprojlim_n \Z/p^n\Z = \Zp.
\]
Now, let $K$ be any number field.
Then $K_{\cyc} = K \cdot \Q_{\cyc}$, i.e., the compositum of $K$ with $\Q_{\cyc}$.
By definition $\Gal(K_{\cyc}/K)\simeq \Zp$, as well.
As before, we write $K_{(n)}$ to denote the $n$-th layer of $K_{\cyc}$ and write $\Gamma_n = \Gal(K_{(n)}/K) \simeq \Z/p^n \Z$.

The Iwasawa algebra $\Lambda=\Lambda(\Gamma)$ is the completed group algebra $\Z_p\llbracket \Gamma \rrbracket :=\varprojlim_n \Z_p[\Gamma/\Gamma^{p^n}]$.
Fix a topological generator $\gamma$ of $\Gamma$; this gives an isomorphism of rings 
\begin{align*}
\Lambda &\xrightarrow{\sim} \Z_p\llbracket x\rrbracket \\
\gamma & \mapsto 1+x.
\end{align*}

Let $M$ be a finitely generated torsion $\Lambda$-module.
The \emph{structure theorem of $\Lambda$-modules} \cite[Theorem~13.12]{Was97} asserts that $M$ is pseudo-isomorphic to a finite direct sum of cyclic $\Lambda$-modules.
In other words, there is a homomorphism of $\Lambda$-modules
\[
M \longrightarrow \left(\bigoplus_{i=1}^s \Lambda/(p^{m_i})\right)\oplus \left(\bigoplus_{j=1}^t \Lambda/(f_j(x)) \right)
\]
with finite kernel and cokernel.
\edit{Moreover, this decomposition is uniquely determined by $M$; see \cite[Corollary~15.19]{Was97}.
We call
the right hand side of the above homomorphism the \textit{elementary module} of $M$ and denote it by $\mathscr{E}(M)$.}
Such a homomorphism is a \emph{pseudo-isomorphism}.
Here, $m_i>0$ and $f_j(x)$ is a distinguished polynomial (i.e.~a monic polynomial with non-leading coefficients divisible by $p$).
The characteristic ideal of $M$ is (up to a unit) generated by the characteristic element,
\[
f_{M}^{(p)}(x) := p^{\sum_{i} m_i} \prod_j f_j(x).
\]
For ease of notation, we often write $f(x) = f_M^{(p)}(x)$.
The $\mu$-invariant of $M$ is defined as the power of $p$ in $f_{M}^{(p)}(x)$.
More explicitly,
\[
\mu(M) = \mu_p(M):=\begin{cases}0 & \textrm{ if } s=0\\
\sum_{i=1}^s m_i & \textrm{ if } s>0.
\end{cases}
\]
The $\lambda$-invariant of $M$ is the degree of the characteristic element, i.e.
\[
\lambda(M) = \lambda_p(M) := \sum_{j=1}^t \deg f_j(x).
\]

\subsection{Selmer group and fine Selmer group of elliptic curves}
Fix an algebraic closure $\overline{\Q}$ of $\Q$.
Then an algebraic extension of $\Q$ is a subfield of this fixed algebraic closure, $\overline{\Q}$.
Throughout, $p$ denotes a fixed rational odd prime.
Let $\EC/\Q$ be an elliptic curve.
Fix a finite set $S$ of primes of $\Q$ containing $p$, the primes of bad reduction of $\EC$, and the unique archimedean prime.
Denote by $\Q_S$, the maximal algebraic extension of $\Q$ unramified outside $S$.
For every (possibly infinite) extension $L$ of $\Q$ contained in $\Q_S$, write $G_S\left({L}\right) = \Gal\left(\Q_S/{L}\right)$.
Write $S\left(L\right)$ for the set of primes of $L$ above $S$.
If $L$ is a finite extension of $\Q$ and $w$ is a place of $L$, we write $L_w$ for its completion at $w$; when $L/\Q$ is infinite, it is the union of completions of all finite sub-extensions of $L$.

\begin{definition}
Fix an elliptic curve $\EC/\Q$ with good reduction at $p$.
View $\EC[p^\infty]$ as a $\Zp$-module equipped with a continuous $G_S(\Q)$-action.
\begin{itemize}
 \item[\textup{(}i\textup{)}] For any finite extension $L/\Q$ set
\begin{align*}
K^1_v\left(\EC/L\right) &= \bigoplus_{w|v} H^1\left(L_w, \EC[p^\infty]\right),\\
J^1_v\left(\EC/L\right) &= \bigoplus_{w|v} H^1\left(L_w, \EC\right)[p^\infty],
\end{align*}
where the direct sum is taken over all primes $w$ of $L$ lying above a rational prime $v$.
\item[\textup{(}ii\textup{)}]For an infinite algebraic extension $\cL/\Q$, define $K_v^1\left(\EC/\cL\right)$ by taking the inductive limit of $K_v^1\left(\EC/L\right)$ over all finite extensions $L/\Q$ contained in $\cL$, and similarly for $J_v^1(\EC/\cL)$.
\item[\textup{(}iii\textup{)}]Let $L$ be an algebraic extension of $\Q$ that is contained inside $\Q_S$, define the \emph{$p^\infty$-fine Selmer group} of $\EC$ over $L$ as
\[
\Sel_0(\EC/L):=\ker\left(H^1(G_S(L),\EC[p^\infty])\longrightarrow\bigoplus_{v\in S(L)} K^1_v(\EC/L)\right).
\]
and the \emph{$p^\infty$-Selmer group} as
\[
\Sel(\EC/L):=\ker\left(H^1(G_S(L),\EC[p^\infty])\longrightarrow\bigoplus_{v\in S(L)} J^1_v(\EC/L)\right).
\]
\end{itemize}
\end{definition}
\noindent Since all primes are finitely decomposed in the cyclotomic $\Zp$-extension, we will henceforth simplify notation and write $\bigoplus_{v\in S(\Q_{\cyc})} H^1\left(\Q_{\cyc,v},\EC[p^\infty]\right)$ in place of $\bigoplus_{v\in S(\Q_{\cyc})} K^1_v(\EC/\Q_{\cyc}))$, and similarly for $\bigoplus_{v\in S(\Q_{\cyc})} J^1_v(\EC/\Q_{\cyc}))$.
Since $p$ is fixed, we drop it from the notation of the (fine) Selmer group.
We remind the reader that both the definitions are independent of $S$ as long as $S$ contains the primes above $p$, the primes of bad reduction of $\EC$, and the archimedean prime(s).

Let $L/\Q$ be an algebraic extension.
The Selmer group sits inside the following short exact sequence
\[
0 \longrightarrow \EC(L) \otimes \Qp/\Zp \longrightarrow \Sel(\EC/L) \longrightarrow \Sha(\EC/L)[p^\infty] \longrightarrow 0.
\]
The left-most object is called the \emph{Mordell--Weil group} of $\EC$ and the right-most object is called the \emph{Shafarevich--Tate group}.
Analogously, the fine Selmer group also sits inside a short sequence
\[
0 \longrightarrow \cM(L) \longrightarrow \Sel_0(\EC/L) \longrightarrow \Zha(\EC/L)[p^\infty] \longrightarrow 0.
\]
The left-most object referred to as the \emph{fine Mordell--Weil group}
\edit{is defined as follows}
\[
\edit{\cM(L) = \ker\left(\EC(L) \otimes \Q_p/\Z_p \longrightarrow \bigoplus_{v\mid p} \EC(L_v) \otimes \Q_p/\Z_p\right).}
\]
The right-most object \edit{which can be defined as the quotient of the fine Selmer group by the fine Mordell--Weil group} is called 
\emph{fine Shafarevich--Tate group}.
An equivalent definition of the fine Selmer group is the following
\begin{equation}
\label{alternative defn of FSG}
0 \longrightarrow \Sel_0(\EC/L) \longrightarrow \Sel(\EC/L) \longrightarrow \bigoplus_{w\mid p} \EC(L_w) \otimes \Qp/\Zp.
\end{equation}

\subsection{Iwasawa theory of elliptic curves}
Given a number field $K$, denote by $K_{(n)}$ the $n$-th layer in the $\Z_p$-extension.
Let $\EC/K$ be an elliptic curve, then (as previously explained) the Selmer group over $K_{\cyc}$ can be viewed as the direct limit of the Selmer groups over the number fields $K_{(n)}$ in $K_{\cyc}$, i.e.
\[
\Sel(\EC/K_{\cyc})=\varinjlim_n \Sel(\EC/K_{(n)}).
\]
Its Pontryagin dual $\Sel(\EC/K_{\cyc})^\vee$ is a finitely generated $\Lambda$-module.

\subsubsection{Good ordinary reduction}
\edit{If $K/\Q$ is an abelian extension and $\EC/K$ has good \emph{ordinary} reduction at all primes above $p$,}
%For $p$ a prime of good \emph{ordinary} reduction and abelian $K/\Q$, 
it is known that $\Sel(\EC/K_{\cyc})$ is $\Lambda$-cotorsion; see \edit{\cite[Theorem~17.4(1)]{Kat04} and \cite[Theorem~1.5]{Gre98}}.
We can therefore attach Iwasawa invariants to $\Sel(\EC/K_{\cyc})^\vee$.
Note that when $K/\Q$ is not abelian, it is still expected that $\Sel(\EC/K_{\cyc})^\vee$ is $\Lambda$-torsion.

\subsubsection{Good supersingular reduction}
\label{prelim: good supersingular}
When \edit{$\EC/\Q$ has good supersingular reduction at $p$ (or at primes above $p$ when base changed)} % $p$ is a prime of good \emph{supersingular} reduction, 
the dual of the Selmer group is provably \emph{not} $\Lambda$-torsion.
To remedy this, we define signed Selmer groups à la S.~Kobayashi \cite{Kob03} and B.~D.~Kim \cite{BDKim13, byoung2014signed}.
For our purposes, we will further suppose that $a_p(\EC) =0$.

\smallskip

\emph{We further suppose that $p$ is unramified in $K$.}

\smallskip

For a fixed prime $p$ and each $n\geq 0$, set $\Kn = K_{(n)} \otimes \Qp$, which is a product of local fields and set $\K_{\cyc} = K_{\cyc} \otimes \Qp$.
We write $\EC(\Kn) = \bigoplus_{v\mid p} \EC(K_{(n),v})$ and similarly for $\EC(\mathbb{K}_{\cyc})$.

Following \cite{Kat21} we define the plus (and minus) norm groups as follows
\begin{align*}
\EC^+(\Kn) &:=
\left\{\mathsf{P}\in \EC(\Kn) \mid \tr_{n/m+1} (\mathsf{P})\in \EC(\K_{m}), \text{ for }0\leq m < n\text{ and }m \text{ even}\right\},\\
\EC^-(\Kn) &:=
\left\{\mathsf{P}\in \EC(\Kn) \mid \tr_{n/m+1} (\mathsf{P})\in \EC(\K_{m}),\text{ for }-1\leq m < n\text{ and }m 
\text{ odd}\right\},
\end{align*}
where $\tr_{n/m+1}:\EC(\Kn)\rightarrow \EC(\K_{m+1})$ is the trace map.
We refer the reader to \cite[Section~2]{BDKim13} for the etymology.
Set $\EC^{\pm}(\K_{\cyc}) = \bigcup_{n\geq 0} \EC^{\pm}(\Kn)$.

\begin{comment}
\DK{[stick to Antonio's definitions and not worry about the standard ones]}
Note that $\EC^\pm(\Kn)\otimes \Qp/\Zp\subset\EC(\Kn)\otimes \Qp/\Zp$; define the plus/minus Selmer group at the $n$-th layer of the $\Z_p$-extension by the short exact sequence 
\[
0 \longrightarrow \Sel^{\pm}(\EC/K_{(n)}) \longrightarrow \Sel(\EC/K_{(n)}) \longrightarrow \frac{\EC(\Kn) \otimes \Qp/\Zp}{\EC^{\pm}(\Kn)\otimes \Q_p/\Z_p}.
\]
Then, the plus (resp.~minus) Selmer groups over $K_{\cyc}$ is defined by taking direct limits,
\[
\Sel^{\pm}(\EC/K_{\cyc}) := \varinjlim_n \Sel^{\pm}(\EC/K_{(n)}).
\]
These Pontryagin dual of these $\pm$-Selmer groups are always expected (and under reasonable hypotheses it is provable) to be $\Lambda$-torsion modules.
\end{comment}

Following \cite{Lei23} we introduce the notion of \emph{plus/minus} Mordell--Weil groups over $\K_{\cyc}$.
Define
\[
\cM^{\pm}(\EC/\K_{\cyc}) := \bigcup_{n\geq 0} \EC^{\pm}(\Kn) \otimes \Qp/\Zp = \EC^{\pm}(\K_{\cyc}) \otimes \Qp/\Zp.
\]
Note that this definition makes sense since $\EC^{\pm}(\Kn)[p^\infty]$ is trivial for all $n\geq 0$.
We remind the reader of a notation in \cite{Kat21} that we will use.
For $F/\Q$ an algebraic extension and a continuous $\Gal(\overline{\Q}/F)$-module $X$, we write $H^i(F \otimes \Q_{p}, X) = \bigoplus_{w\mid p} H^i(F_w, X)$.

Since $\EC(\K_{\cyc})[p^\infty]$ is always trivial; see \cite[Proposition~3.6]{Kat21}, it now easily follows that the inflation-restriction exact sequence is
\[
H^1\left( \K_{\cyc}, \EC[p^\infty]\right)^{\Gamma_n} \simeq H^1\left( \Kn, \EC[p^\infty]\right).
\]
Similarly, we have 
\[
H^1\left( \K_{\cyc}, \EC[p^r]\right) \simeq H^1\left( \Kn, \EC[p^\infty]\right)[p^r] 
\]
induced by the short exact sequence
\[
0 \longrightarrow \EC[p^r] \longrightarrow \EC[p^\infty] \xrightarrow{\times p^r} \EC[p^\infty] \longrightarrow 0.
\]

Now define $\cM^{\pm}(\EC/\Kn) = \cM^{\pm}(\EC/\K_{\cyc})^{\Gamma_n}$.
For a fixed positive integer $k$, we will often work with 
\[
\cM_{p^k}^{\pm}(\EC/\Kn) := \cM^{\pm}(\EC/\Kn)[p^k].
\]
Finally, we define the $p^k$ and $p$-primary plus/minus Mordell--Weil groups of $\EC$ over $K_{(n)}$ as follows
\begin{align*}
\cM_{p^k}^{\pm}(\EC/K_{(n)}) &:= \ker\left(\EC(K_{(n)})/p^k \edit{\xrightarrow{\kappa}} \frac{H^1(\Kn, \EC[p^k])}{\cM_{p^k}^{\pm}(\EC/\Kn)} \right)\\
\cM^{\pm}(\EC/K_{(n)}) &:= \varinjlim_k \cM_{p^k}^{\pm}(\EC/K_{(n)}) = \ker\left(\EC(K_{(n)}) \otimes \Qp/\Zp \xrightarrow{\kappa} \frac{H^1(\Kn, \EC[p^\infty])}{\cM^{\pm}(\EC/\Kn)} \right)\\
& \hspace{1.27in} = \ker\left(\EC(K_{(n)}) \otimes \Qp/\Zp \longrightarrow \frac{H^1(\Kn, \EC[p^\infty])}{\cM^{\pm}(\EC/\K_{\cyc})} \right).
\end{align*}
\edit{In the above definitions, we write $\kappa$ to denote the Kummer map.}
For the last equality we make the following observation: $\image(\kappa)$ is fixed by $\Gamma_n$.
Hence, by definition
\[
\image(\kappa) \cap \cM^{\pm}(\EC/\K_{\cyc}) \subseteq \cM^{\pm}(\EC/\K_{\cyc})^{\Gamma_n} = \cM^{\pm}(\EC/\Kn).
\]

By taking direct limits over $n$, we get the appropriate definitions of the plus/minus Mordell--Weil groups of $\EC$ over $K_{\cyc}$.
In other words,
\[
\cM^{\pm}(\EC/K_{\cyc}) = \varinjlim_n \cM^{\pm}(\EC/K_{(n)}).
\]

\begin{remark}
We use the notation $\cM^{\pm}(\EC/F)$ where $F$ is a field for $\pm$ Mordell--Weil groups, as introduced in \cite{Lei23}.
We emphasize that this is not a variant of the \emph{fine} Mordell--Weil group.
\end{remark}

We can introduce the plus/minus Selmer groups over $K_{(n)}$ (and hence also $K_{\cyc}$):
\[
\Sel^{\pm}(\EC/K_{(n)}) := \ker\left( H^1(K_{(n)}, \EC[p^\infty]) \longrightarrow \frac{H^1(\Kn, \EC[p^\infty])}{\cM^{\pm}(\EC/\Kn)} \times \prod_{w\nmid p} H^1(K_{(n),w}, \EC[p^\infty])\right).
\]
The plus/minus Selmer groups over $K_{\cyc}$ \edit{are} defined by taking direct limits,
\[
\Sel^{\pm}(\EC/K_{\cyc}) := \varinjlim_n \Sel^{\pm}(\EC/K_{(n)}).
\]
These Pontryagin dual\edit{s} of these $\pm$-Selmer groups are always expected (and under reasonable hypotheses it is provable) to be $\Lambda$-torsion modules.

\begin{remark}
Sometimes we will be required to discuss $\Sel(\EC/K_{\cyc})$ for $p$ a prime of good ordinary reduction and $\Sel^{\pm}(\EC/K_{\cyc})$ for $p$ a prime of good supersingular reduction, simultaneously.
In such cases, we use the notation $\Sel^{\dagger}(\EC/K_{\cyc})$ with $\dagger\in \{\emptyset, \pm \}$; here, by $\Sel^{\emptyset}(\EC/K_{\cyc})$ we simply mean $\Sel(\EC/K_{\cyc})$ and $p$ is good ordinary.
\end{remark}

\subsection{Additional notation}
\begin{definition}
\label{additional notation}
Let $A$ be a $\Zp$-module.
\begin{enumerate}
\item[\textup{(}i\textup{)}] The \emph{Pontryagin dual} of $A$, denoted by $A^\vee$ is defined as
\[
A^\vee = \operatorname{Hom}(A, \Qp/\Zp)
\]
where $\operatorname{Hom}$ means the group of continuous homomorphisms.
\item[\textup{(}ii\textup{)}] The $\Zp$-dual of $A$, denoted by $A^D$ is defined as
\[
A^D = \operatorname{Hom}(A, \Zp).
\]
\item[\textup{(}iii\textup{)}] The \emph{Tate module} of $A$, denoted by $T_p( A)$, is defined via a projective limit where the connecting maps are given by $\times p$; i.e.,
\[
T_p(A) = \varprojlim_k A[p^k].
\]
\item[\textup{(}iv\textup{)}] The \emph{$p$-adic completion} of $A$, denoted by $A^*$, is defined via a projective limit where the connecting maps are given by projections; i.e.,
\[
A^* = \varprojlim_k A/p^kA.
\]
\item[\textup{(}v\textup{)}] The $\Qp$-linear versions of $T_p( A)$ and $A^*$ are defined as
\[
V_p( A) = T_p( A) \otimes_{\Zp} \Qp, \ \quad \ A^{\bullet} = A^* \otimes_{\Zp} \Qp.
\]
\end{enumerate}
\end{definition}

\section{Algebraic preliminaries}
\label{sec: algebraic preliminary}
Fix an abelian group $G$.
By the structure theorem (of finitely generated modules over principal ideal domains), we know that
\[
G \simeq \Z/p_1^{n_1}\Z \times \ldots \times \Z/p_r^{n_r}\Z \text{ where each }p_i \text{ is a prime and }n_i\in \Z_{>0}.
\]
For ease of notation, set $G^{(i)} = \Z/p_i^{n_i}\Z$ and write $\sigma_i$ to denote its generator.
Then,
\[
\Q[G] \simeq \bigotimes_{i=1}^r \Q[G^{(i)}]
\]
as a $G$-module with a natural $G$-action.

\begin{definition}
For any integer $n>0$ and variable $x$, define the formal polynomial
\[
\omega(x,n) = \frac{x^n -1}{x-1}.
\]
When $n$ is a prime, the formal polynomial is irreducible over $\Q$.
\end{definition}

We begin the discussion by considering one cyclic piece of $G$.
\edit{Note that} as a $G^{(i)}$-module we have
\[
\Q[G^{(i)}] \simeq \Q \oplus \bigoplus_{j = 1}^{n_i} \frac{\Q[G^{(i)}]}{\langle \omega(\sigma_i^{p_i^{j - 1}}, p_i)\rangle }.
\]
\edit{For a proof, see \cite[\S21, Exercise~1]{Curtis-Reiner-81}.} Here, we may view each component as an irreducible representation over $\Q$; we denote them as $V_{i,j}$ for $1\leq j \leq n_i$.
Indeed, this is because
\[
\frac{\Q[G^{(i)}]}{\langle \omega(\sigma_i^{p_i^{j - 1}}, p_i) \rangle} \simeq \frac{\Q[x]}{\langle x^{p_i^{n_i}}-1, \ \omega(x^{p_i^{j - 1}}, p_i) \rangle} \simeq \frac{\Q[x]}{\langle \omega(x^{p_i^{j - 1}}, p_i) \rangle} .
\]
For consistency of notation, set $\Q = V_{i,0}$ when viewed as an irreducible representation (over $\Q$).

In the general case, consider the following collection of vectors 
\[
\A(G) = \A := \{\vec{\boldsymbol{\alpha}} = (\alpha_1, \ldots, \alpha_r) \colon 0 \leq \alpha_i \leq n_i \text{ and } \alpha_i \in \Z\}
\]
\edit{and give partial ordering $\preccurlyeq$ as follows: for two elements $\vec{\boldsymbol{\beta}} = (\beta_1, \ldots, \beta_r)$ and $\vec{\boldsymbol{\alpha}} = (\alpha_1, \ldots, \alpha_r)$ in $\A$, we say  $\vec{\boldsymbol{\beta}} \preccurlyeq \vec{\boldsymbol{\alpha}}$ if $\beta_i \leq \alpha_i$ for all $1\leq i \leq r$.
}
Now for each $\val\in \A$, define the $G$-module
\[
W_{\val} := \bigotimes_{i=1}^r V_{i, \alpha_i} \text{ where } \val = (\alpha_i)_i.
\]

$\{W_{\val} \colon \val \in \A\}$ is the collection of  all irreducible representations of $G$ over $\Q$.
Thus, the irreducible representations can be parameterized by the tuples $\val$ and $\Q[G]$ is the direct sum of those $W_{\val}$.
The reason why we are required to find all of them is, we need to know how many of each of them appear in each $\EC(K_{(n)}) \otimes \Q$.
This pinpoints $\EC(K_{(n)}) \otimes \Q$ as a $G$-representation.

Since each $V_{i,\alpha_i}$ is irreducible (as explained previously), it follows that $W_{\val}$ is an irreducible $G$-representation over $\Q$ and
\[
\Q[G] \simeq \bigoplus_{\val\in \A} W_{\val}.
\]

We are sometimes required to work with the \emph{integral modules}.
We introduce some notation:
\[
\overline{V}_{i,j} = \frac{\Z[G^{(i)}]}{\langle \omega(\sigma_i^{p_i^{j - 1}}, p_i) \rangle} \text{ and } \overline{W}_{\val} = \bigotimes_{i=1}^r \overline{V}_{i, \alpha_i}.
\]
Note that
\[
V_{i,j} = \overline{V}_{i,j} \otimes \Q \text{ and } W_{\val} = \overline{W}_{\val}\otimes \Q.
\]

Now, suppose that the following hypothesis holds:
\vspace{0.2cm}
\begin{itemize}
\item[\mylabel{star}{\textup{(}$\star$\textup{)}}] Set $\edit{\exp}(G) = m$ and suppose that $p$ is a prime which does not split in $\Q(\zeta_m)/\Q$.
\end{itemize}
\vspace{0.2cm}
Recall that $p$ is ramified in $\Q(\zeta_m)$ precisely when $p\mid m$.
On the other hand for $p\nmid m$, if $\phi(m)$ is the smallest integer such that $p^{\phi(m)} \equiv 1 \pmod{m}$ then $p$ is non-split in $\Q(\zeta_m)$.
Now, suppose $m = p^r m'$ with $\gcd(p,m') =1$.
Then \ref{star} holds iff $\phi(m\edit{'})$ is the smallest integer satisfying $p^{\phi(m')} \equiv 1 \pmod{m'}$.
In particular, this means that $p$ must be a generator of $(\Z/m'\Z)^{\times}$; i.e., $m'$ can be either $2$, $4$, $\ell^r$ or $2\ell^r$ for a prime $\ell \neq p$ and $r\geq 0$.

\begin{lemma}
If \ref{star} holds, the formal polynomial $\omega(x^{p_i^{j-1}}, p_i)$ is irreducible in $\Qp[x]$.
\end{lemma}

\begin{proof}
Write $\ell$ to denote any prime.
Note that the formal polynomial $\omega(x^{\ell^{j-1}}, \ell)$ is irreducible in $\Qp[x]$ precisely when it is the minimal polynomial of $\zeta_{\ell^{j}}$ in $\Qp[x]$.
Equivalently, we require that 
\[
\deg(\omega(x^{\ell^{j-1}}, \ell)) = [\Q(\zeta_{\ell^j}): \Q] =[\Qp(\zeta_{\ell^j}): \Qp].
\]
The second equality holds if and only if $p$ is non-split in $\Q(\zeta_{\ell^j})$ and this is ensured by \ref{star}.
\end{proof}

Thus, $V_{i,j} \otimes \Qp$ is an irreducible $G^{(i)}$-representation over $\Qp$; hence $W_{\val} \otimes \Qp$ is an irreducible $G$-representation over $\Q_{\edit{p}}$.

For each $\val = (\alpha_i)_i \in \A$, define an $\val$-subgroup of $G$ as follows
\[
G_{\val} := {G^{(1)}}^{p_1^{\alpha_1}} \times \ldots \times {G^{(r)}}^{p_r^{\alpha_r}}.
\]
When $\val = (n_i)_i$, i.e., $\val$ is \emph{maximal} then $G_{\val} = \{1\}$.
Whereas, when $\val = (0)_i$ then $G_{\val} = G$.

Upon fixing $\val_0 \in \A$, we obtain a (unique) subset
\[
\{W_{\val} \colon \val \preccurlyeq \val_0\} \subseteq \{W_{\val} \colon \val \in \A\}.
\]
This subset consists of $G$-modules which are fixed by $G_{\val_0}$.

In the next lemma we show that $W_{\veb}$ is the unique irreducible representation of $G$ over $\Q$, which is fixed by  precisely by the groups $G_{\val}$ for $\veb \preccurlyeq \val$.

\begin{lemma}
\label{lemma 1 for Prop 2 of note 1}
Suppose that $\val, \veb\in \A$.
Then
\[
W_{\veb}^{G_{\val}} = \begin{cases}
W_{\veb} & \text{ if } \veb \preccurlyeq \val\\
0 & \text{ otherwise}.
\end{cases}
\]
\end{lemma}

\begin{proof}
Consider the element
\[
\boldsymbol{w} = w_1 \otimes \ldots \otimes w_r \in W_{\veb},
\]
where each $w_i \in V_{i, \beta_i}$.
Also, let 
\[
g = (g_1 , \ldots, g_r)\in G_{\val}
\]
where each $g_i \in G^{{(i)}^{p_i^{\alpha_i}}}$.
Hence, each $g_i = \sigma_i^{k_{\edit{i}}p_i^{\alpha_i}}$ for some positive integer $k_{\edit{i}}$.
Via the isomorphism
\[
V_{i,\beta_i} \simeq \frac{\Q[x]}{\langle \omega(x^{p_i^{\beta_i - 1}}, p_i) \rangle}, 
\]
the action of $\sigma_i$ on $V_{i,\beta_i}$ corresponds to multiplication by $x$.
Thus, $\sigma_i^{k_{\edit{i}}p_i^{\alpha_i}}$ corresponds to $\left(x^{{p_i}^{\alpha_i}}\right)^{k_{\edit{i}}}$.

By definition, if $\veb\preccurlyeq \val$ then $\beta_i \leq \alpha_i$ for each $i$.
Thus
\[
\omega(x^{p_i^{\beta_i - 1}}, p_i) \mid x^{p_i^{\alpha_i}} - 1
\]
and
\[
\left( x^{p_i^{\alpha_i}} \right)^{k_{\edit{i}} }\equiv 1 \pmod{\omega(x^{p_i^{\beta_i - 1}}, p_i)}
\]
This means $g_i \cdot w_i = w_i$ for all $i$ and $g \cdot \boldsymbol{w} = \boldsymbol{w}$ as desired.

Next suppose that $\veb\not\preccurlyeq \val$ and that $\boldsymbol{w}\in W_{\veb}^{G_{\val}}$.
It is clear from the definition of partial ordering that there exists some $j$ such that $\beta_j > \alpha_j$.
Consider the element
\[
h = (1, 1, \ldots, 1, \sigma_j^{p_j^{\alpha_j}}, 1, \ldots, 1)\in G_{\val}.
\]
Note that $\ord(h) = p_j^{n_j - \alpha_j}$.
Write $\mathsf{N}(h)$ to denote the norm element, i.e.,
\[
\mathsf{N}(h) = 1 + h + h^2 + \ldots + h^{p_j^{n_j - \alpha_j}-1} \in \Q[G].
\]
Then, 
\begin{equation}
\label{eqn lemma 2.2}
\mathsf{N}(h)\cdot \boldsymbol{w} = p_j^{n_j - \alpha_j}\boldsymbol{w} \edit{.}
\end{equation}
On the other hand if we consider \edit{an arbitrary} element $\boldsymbol{w'}\in W_{\veb}$ then
\[
\mathsf{N}(h)\cdot \boldsymbol{w'} = w'_1 \otimes \ldots \otimes \sum_{k_{\edit{j}}=0}^{p_j^{n_j - \alpha_j}-1} \sigma_j^{k_{\edit{j}} p_j^{\alpha_j}}\cdot w'_j\otimes \ldots \otimes w'_r ,
\]
But, $\displaystyle\sum_{k_{\edit{j}}=0}^{p_j^{n_j - \alpha_j}-1} \sigma_j^{k_{\edit{j}} p_j^{\alpha_j}}$ acts on $V_{j,\beta_j} \simeq \frac{\Q[x]}{\langle \omega(x^{p_j^{\beta_j - 1}}, p_j) \rangle}$ as
\[
1 + x^{p_j^{\alpha_j}} + x^{2p_j^{\alpha_j}} + \ldots + x^{p_j^{n_j} - p_j^{\alpha_j}} = \frac{x^{p_j^{n_j}}-1}{x^{p_j^{\alpha_j}}-1}.
\]
But, $n_j \geq \beta_j > \alpha_j$, so $\omega(x^{p_j^{\beta_j - 1}}, p_j)$ divides $\frac{x^{p_j^{n_j}}-1}{x^{p_j^{\alpha_j}}-1}$.
Thus, $\mathsf{N}(h)\cdot \boldsymbol{w'}=0$.
By \eqref{eqn lemma 2.2}, it follows that $\mathsf{N}(h)\cdot \boldsymbol{w}\edit{=0}$ and hence $\boldsymbol{w}=0$.
\end{proof}

\section{Rank Growth Data and Fine Selmer Groups}
\label{sec: Rank Growth Data and Fine Selmer Groups}

\subsection{Arithmetic application of algebraic results}
Fix an elliptic curve $\EC/\Q$ with good reduction at a fixed odd prime $p$.
Let $K$ be an abelian extension over $\Q$ with $\Gal(K/\Q)\simeq \sG$ which is disjoint from $\Q_{\cyc}$.
Recall that each $n$-th layer of the cyclotomic $\Zp$-extension of $K$ is denoted by $K_{(n)}$.
For each $n\geq 0$, set $\Kn = K_{(n)} \otimes \Qp$, which is a product of local fields.
In fact, $K_{(n)}$ is Galois over $\Q$ and we set
\[
\mathcal{G}_n :=\Gal(K_{(n)}/\Q) = \sG \times \Gamma_n.
\]
In the limit $K_{\cyc}/\Q$ is also a Galois extension, and we denote this Galois group by $\mathcal{G} = \sG \times \Gamma$.

\begin{lemma}
\label{Lemma 1 note 1}
Fix a prime $p$ of good reduction of an elliptic curve $\EC/\Q$.
Suppose that $p\nmid a_p$, i.e., $p$ is a prime of good \emph{ordinary} reduction.
As $\cG_n$-modules, the following isomorphism is true
\[
\EC(\Kn)^{\bullet} \simeq \Qp[\cG_n].
\]
\end{lemma}

We note that when we write $\EC(\Kn)$ we mean $\oplus_{\fp\mid p} \EC(K_{(n),\fp})$

\begin{proof}
By standard results on the structure of elliptic curves over local fields, we have
\[
0 \longrightarrow \hat{\EC}(\fm_{n,\fp}) \longrightarrow \EC(K_{(n),\fp}) \longrightarrow \text{ finite group } \longrightarrow 0,
\]
where $\hat{\EC}$ denotes the formal group of $\EC$ at $\fp$ and $\fm_{n,\fp}$ is the maximal ideal of the ring of integers of $K_{(n),\fp}$.
Upon tensoring the above short exact sequence by $\Qp$ and using \cite[Proposition~VII.6.3]{Sil86} yields the result.
\end{proof}

\begin{proposition}
\label{Prop 2 note 1}
Let $L/F$ be a Galois extension with $\Gal(L/F)=G$ and let $\EC/F$ be an elliptic curve.
\edit{Then there exist non-negative integers $e_{\val}$
for $\val\in \mathbb{A}$ such that a}s $G$-modules, the following isomorphism is true
\[
\EC(L) \otimes \Q \simeq \bigoplus_{\val\in \A} W_{\val}^{e_{\val}},
\]
where $e_{\vec{0}} = \rk_{\Z} \EC(F)$ and $\rk_{\Z} \EC(L^{G_{\val}}) = \sum_{\vec{\boldsymbol{\beta}} \preccurlyeq \val} e_{\vec{\boldsymbol{\beta}}}\dim W_{\vec{\boldsymbol{\beta}}}$.
\end{proposition}

\begin{proof}
This result is obtained by studying the $G$-action on the Mordell--Weil group $\EC(L)$.
More precisely, observe that the natural map
\[
\left( \EC(L) \otimes \Q\right)^{G_{\val}} \longrightarrow \EC(L)^{G_{\val}} \otimes \Q
\]
is bijective.
Thus,
\[
\dim_{\Q}\left( \EC(L) \otimes \Q\right)^{G_{\val}} = \rk_{\Z} \EC(L^{G_{\val}}).
\]
The result now follows from Lemma~\ref{lemma 1 for Prop 2 of note 1}.
\end{proof}

\begin{definition}
The \emph{$p^n$-cyclotomic polynomial} is defined as
\begin{align*}
\Phi_n(x) := 
\begin{cases}
\frac{(1+x)^{p^n} - 1}{(1+x)^{p^{n-1}} - 1} & \text{ if } n\geq 1\\
x & \text{ if } n=0.
\end{cases}
\end{align*}    
\end{definition}

With this notation, we can identify
\[
\Zp[\Gamma_n] \simeq \bigoplus_{k=0}^n \frac{\Zp[x]}{\langle \Phi_k(x)\rangle}.
\]
Thus,
\[
\Qp[\Gal(K_{(n)}/\Q)] = \Qp[\cG_n] \simeq \Qp[\sG \times \Gamma_n] \simeq \Qp[\sG]\otimes \left(\bigoplus_{k=0}^n\frac{\Qp[x]}{\langle \Phi_k(x)\rangle}\right).
\]

In view of Proposition~\ref{Prop 2 note 1} and the above discussion, we note that for any positive integer $k$, we can associate \edit{with $\EC$} a quantity $e_{\val,k}$ which is the number of $W_{\val} \otimes \frac{\Q[x]}{\langle \Phi_k(x)\rangle}$ components that appear in $\EC(K_{(k)}) \otimes \Q$.
We make this more precise through the following proposition.

\begin{proposition}
\label{main result from note clarifying remark}
For a fixed non\edit{-}negative integer $n$, consider the following isomorphism of $\mathcal{G}_n$-modules
\begin{equation}
\label{eq: isomorphism to answer ref}
\EC(K_{(n)}) \otimes \Q \simeq \bigoplus_{\val\in \A} \bigoplus_{k=0}^n \left( W_{\val} \otimes \frac{\Q[x]}{\langle \Phi_k(x)\rangle} \right)^{e_{\val,k}},
\end{equation}
where $e_{\val,k}$ is a non-negative integer.
Set $e_{\val,n}= e_{\val}$\edit{.} 
Then for each $\val\in \A(\sG)$, the integer $e_{\val}$ is given by the following inductive formula

\[
\frac{1}{\dim W_{\val}}\left(\frac{1}{\phi(p^n)}\left( \rk_{\Z}\EC(K_{(n)}^{\sG_{\val}}) - \rk_{\Z}\EC(K_{(n-1)}^{\sG_{\val}}) \right) - \sum_{\vec{\boldsymbol{\beta}} \prec \val} e_{\vec{\boldsymbol{\beta}}}\dim W_{\vec{\boldsymbol{\beta}}} \right)= e_{\val}
\]
where $\EC(K_{(-1)})$ is the trivial group.
\end{proposition}

\begin{proof}
Let $L= K_{(n)}$ and $F=\Q$ in Proposition~\ref{Prop 2 note 1} and $G_{\val} = \sG_{\val} \times \{1\} \subset \sG \times \Gamma_n$.
Then
\[
\rk_{\Z} \EC(K_{(n)}^{G_{\val}}) = \rk_{\Z} \EC(K_{(n)}^{\sG_{\val} \times \{1\}}) = \sum_{\substack{\veb \preccurlyeq \val\\ 0 \leq k \leq n}} e_{\veb, k} \dim \left( W_{\veb} \otimes \frac{\Q[x]}{\langle \Phi_k(x)\rangle}\right).
\]
Now, suppose we let $L= K_{(n)}$ in Proposition~\ref{Prop 2 note 1} and $G_{\val} = \sG_{\val} \times \Gamma_n^{p^{n-1}} \subset \sG \times \Gamma_n$.
Then
\[
\rk_{\Z} \EC(K_{(n-1)}^{\sG_{\val}}) = \rk_{\Z} \EC(K_{(n)}^{\sG_{\val} \times \Gamma_n^{p^{n-1}}}) = \sum_{\substack{\veb \preccurlyeq \val\\ 0 \leq k \leq n-1}} e_{\veb, k} \dim \left( W_{\veb} \otimes \frac{\Q[x]}{\langle \Phi_k(x)\rangle}\right).
\]
It now follows easily that
\begin{align*}
\rk_{\Z} \EC(K_{(n)}^{\sG_{\val}}) - \rk_{\Z} \EC(K_{(n-1)}^{\sG_{\val}}) & = \sum_{\veb \preccurlyeq \val} e_{\veb, n} \dim \left( W_{\veb} \otimes \frac{\Q[x]}{\langle \Phi_n(x)\rangle}\right)\\
& = \sum_{\veb \preccurlyeq \val} e_{\veb, n} \dim W_{\veb} \edit{\cdot} \left(p^{n-1}(p-1)\right)\\
& = \phi(p^n) \left( e_{\val}\dim(W_{\val}) + \sum_{\veb \prec \val} e_{\veb} \dim(W_{\veb})\right).
\end{align*}
The result is now immediate.
\end{proof}

When $K=\Q$ and $\cG_n = \Gamma_n \simeq \Z/p^n\Z$, we note that  $\val = \vec{0}$.
It is straight-forward to note that $W_{\vec{0}}$ is trivial.
This essentially recovers the calculations done in \cite[p.~14 just before Lemma~4.2]{Lei23}.

\begin{definition}
An $\val\in \A$ is called \emph{$n$-positive }if $e_{\val,n} = e_{\val}(\EC/K_{(n)})>0$.
\end{definition}

We make a quick observation.
Suppose that $n$ is a fixed positive integer and $\val = (n_i)_i \in \A$ is maximal.
Then, in view of Proposition~\ref{main result from note clarifying remark} we have that
\[
\frac{1}{p^{n-1}(p-1)} \left( \rk_{\Z}\EC(K_{(n)}) - \rk_{\Z}\EC(K_{(n-1)}) \right) = \sum_{\substack{\val\in \A \\ n-\text{positive}} }e_{\val,n}\dim W_{\val}.
\]

\begin{lemma}
\label{lemma 3 note 1}
Fix a positive integer $n$ and suppose that \ref{star} holds for the group $G = \sG = \Gal(K/\Q)$.
Consider the natural map
\[
\eta_n : \EC(K_{(n)})^{\bullet} \longrightarrow \EC(\Kn)^{\bullet}
\]
induced by the injection $\EC(K_{(n)})\hookrightarrow \EC(\Kn)$.
Then
\[
\image(\eta_n)[\Phi_n(x)] \simeq \left( \bigoplus_{\substack{\val\in \A\\
n-\text{positive}}} W_{\val}\right) \otimes  \frac{\Qp[x]}{\langle \Phi_n(x)\rangle}.
\]
\end{lemma}

Before we prove the result, we want to point out that the right hand side of the isomorphism in the second assertion depends on the elliptic curve as the direct sum is over $\val\in \A$ with $n$-positive, and the notion of `positivity' captures information about the elliptic curve.

\begin{proof}
We start by studying the co-domain of the function $\eta_n$: 
\begin{align*}
\EC(\Kn)^{\bullet} & \simeq \Qp[\cG_n] \hspace{2.5in} \text{ by Lemma~\ref{Lemma 1 note 1}}\\
& \simeq \Qp[\cG_{n-1}] \oplus \frac{\Qp[\sG][x]}{\langle \Phi_n(x)\rangle}\\ 
& \simeq \EC(\K_{(n-1)})^{\bullet} \oplus \frac{\Qp[\sG][x]}{\langle \Phi_n(x)\rangle} \hspace{1.5in} \text{ by Lemma~\ref{Lemma 1 note 1}} \\
&\simeq \EC(\K_{(n-1)})^{\bullet} \oplus \left( \left(\bigoplus_{\val} W_{\val}\right) \otimes \frac{\Qp[x]}{\langle \Phi_n(x)\rangle}\right).
\end{align*}
Thus,
\[
\EC(\K_{(n)})^{\bullet}[\Phi_n(x)] \simeq \left(\bigoplus_{\val} W_{\val}\right) \otimes \frac{\Qp[x]}{\langle \Phi_n(x)\rangle}.
\]
We note that the above isomorphism shows that the $\Phi_n(x)$-torsion part of the co-domain does not `see' the elliptic curve in any way.

Suppose that $\val\in \A$ is $n$-positive; i.e., there exists a non-torsion point $\mathsf{P}\in \EC(K_{(n)})\setminus \EC(K_{(n-1)})$ such that $\mathsf{P}$ lies in some component $W_{\val} \otimes \frac{\Q[x]}{\langle \Phi_n(x)\rangle}$.
Therefore, $\mathsf{P}\in \EC(\Kn)\setminus \EC(\K_{(n-1)})$.
This gives a non-trivial element in $\image(\eta_n)[\Phi_n(x)]$.
Here, we are using \ref{star}.
Since $W_{\val} \otimes \frac{\Qp[x]}{\langle \Phi_n(x)\rangle}$ is an irreducible $\cG_n$-representation over $\Qp$; it follows that
\[
W_{\val} \otimes \frac{\Qp[x]}{\langle \Phi_n(x)\rangle} \subseteq \image(\eta_n)[\Phi_n(x)].
\]

On the other hand, if $\val\in \A$ is \emph{not} $n$-positive, then $\EC(\K_{(n)})^{\bullet}[\Phi_n(x)]$ does not contain any component $W_{\val} \otimes \frac{\Qp[x]}{\langle \Phi_n(x)\rangle}$, and neither does $\image(\eta_n)[\Phi_n(x)]$.
This completes the proof.
\end{proof}

Before continuing, we note that for a fixed positive integer $n$, the following isomorphisms hold:
\begin{align}
\label{1}
\EC(K_{(n)}) \otimes \Qp & \simeq \bigoplus_{k=0}^n \bigoplus_{\substack{\val\in \A\\ k-\text{positive}}} \left( W_{\val}^{e_{\val,k}} \otimes \frac{\Qp[x]}{\langle \Phi_k(x)\rangle} \right)\\
\label{2}
\image(\eta_n) & \simeq \bigoplus_{k=0}^n \bigoplus_{\substack{\val\in \A\\ k-\text{positive}}} \left( W_{\val} \otimes \frac{\Qp[x]}{\langle \Phi_k(x)\rangle} \right).
\end{align}

\subsection{\texorpdfstring{$\Lambda$}{}-module structure of the fine Mordell--Weil group}
We relate the \emph{rank growth data} of $\EC$ in each layer of the extension $K_{\cyc}/\Q$ to the structure of the fine Mordell--Weil group $\cM(\EC/K_{\cyc})$.
Further, if we assume that the fine Shafarevich--Tate group is finite over $K_{\cyc}$, then the \emph{rank growth data} completely determines the structure of the fine Selmer group $\Sel_0(\EC/K_{\cyc})$.

In \cite[Theorem~3.3]{Lei23}, Lei has shown that for some integers $c_i$ and $u$, the $\Lambda$-module structure of the fine Mordell--Weil group is as follows
\[
\cM(\EC/K_{\cyc})^\vee \sim \bigoplus_{i=1}^{u} \frac{\Lambda}{\langle \Phi_{c_i}(x) \rangle}.
\]
Our first main result is to make the above result more precise.
In other words, we count the number of times $\Phi_n(x)$ arises in the structure of $\cM(\EC/K_{\cyc})^\vee$.
For this, we need to introduce a definition.

\begin{definition}
\label{theta n}
Fix a non-negative integer $n$ and consider the $n$-th layer of the $\Zp$-extension of $K$ such that $\Gal(K_{(n)}/\Q)\simeq \cG_n$.
Define
\begin{align*}
\theta_n = \theta(\cG_n) &:= \dim_{\Qp}\left( \bigoplus_{\substack{\val\in \A \\ n-\text{positive}}} W_{\val}\right) = \sum_{\substack{\val\in \A \\ n-\text{positive}}} \dim_{\Qp}(W_{\val})\\
e_n &:= 
\frac{\rk_{\Z}\EC(K_{(n)}) - \rk_{\Z}\EC(K_{(n-1)})}{\phi(p^n)}
\end{align*}
\end{definition}

\begin{remark}
When $K=\Q$ and $\cG_n = \Gamma_n$, we know that $\val = 0$.
This means that $\theta_n$ is either 1 (or 0) depending on whether $0$ is (or is not) $n$-positive.
But, this latter condition is completely determined by whether $e_n$ is 0 or not.
In other words we obtain $\theta_n = 0$ (resp.~1) when $e_n = 0$ (resp.~positive).
\end{remark}

We think of $e_n$ as an indicator of whether there is rank jump from the $(n-1)$-th layer to the $n$-th layer.
Whereas, $\theta_n$ measures the growth of $p$-adic image of $\eta_n$, due to the rank growth at the $n$-th layer.
We now rewrite the statement of Theorem~\ref{thm A} and give its proof.

\begin{theorem}
\label{main result fine Selmer} %Analogue of 3.8?
Let $K/\Q$ be an abelian extension which is disjoint from $\Q_{\cyc}$ and suppose that \ref{star} is satisfied for $\Gal(K/\Q) = \mathsf{G}$.
Let $\EC/\Q$ be an elliptic curve with good reduction at $p\geq 3$.
Suppose that $\Zha(\EC/K_{(n)})[p^\infty]$ is finite for each $n$.
Then the $\Lambda$-module structure of the dual fine Mordell--Weil group is given by
\[
\cM(\EC/K_{\cyc})^\vee \sim \bigoplus_{n\geq 0} \left(\frac{\Lambda}{\langle\Phi_n(x)\rangle}\right)^{e_n - \theta_n}.
\]
In particular, the characteristic ideal of the dual fine Mordell--Weil group is given by
\[
\charac_{\Lambda}(\cM(\EC/K_{\cyc})^\vee) = \prod_{n\geq 0} \langle (\Phi_n(x))^{e_n -\theta_n} \rangle.
\]
\end{theorem}

\begin{proof}
Fix an integer $n$.
We start with the exact sequence
\begin{equation}
\label{SES proof thm 1}
0 \longrightarrow T_p\cM(\EC/K_{(n)}) \otimes \Qp \longrightarrow \EC(K_{(n)})^{\bullet} \xrightarrow{\eta_n} \EC(\Kn)^{\bullet}
\end{equation}
We analyze the first two terms of the above exact sequence.
First, we observe that by applying \cite[Corollary~3.8]{Lei23}, we get the following isomorphism of $\Qp[\Gamma_n]$-modules \edit{for some integers $s_t$}
\begin{equation}
\label{st}
T_p\cM(\EC/K_{(n)}) \otimes \Qp = V_p\cM(\EC/K_{(n)}) \simeq \bigoplus_{t= 0}^n \left( \frac{\Qp[x]}{\langle\Phi_{t}(x)\rangle}\right)^{s_t}.
\end{equation}

\noindent \emph{Claim:} As $\Lambda$-modules the following isomorphism holds:
\[
\EC(K_{(n)}) \otimes \Qp \simeq \bigoplus_{t=0}^n \left( \frac{\Qp[x]}{\langle\Phi_{t}(x)\rangle}\right)^{e_t}.
\]

\noindent \emph{Justification:} We apply Proposition~\ref{Prop 2 note 1} for $F=K$ and $L=K_{(n)}$ and $G =\Gal(K_{(n)}/K) = \Gamma_n$.
As $\Gamma_n$-modules, we have the following isomorphism 
\[
\EC(K_{(n)}) \otimes \Qp \simeq \bigoplus_{t=0}^n \left( \frac{\Qp[x]}{\langle\Phi_{t}(x)\rangle}\right)^{m_t} \text{ for some integers } m_t.
\]
By convention, note that $m_0 = e_0$.
For each integer $1\leq i \leq n$, we observe that
\[
\rk_{\Z}\EC(K_{(i)}) = \rk_{\Z}\EC(K_{(n)}^{\Gamma_n^{p^i}}) = m_0 + \sum_{k=1}^{i} m_k (p^{k-1}(p-1)).
\]
Now we apply this formula with $i=t$ and $i=t-1$, we get
\[
\rk_{\Z}\EC(K_{(t)}) - \rk_{\Z}\EC(K_{(t-1)}) = m_t(p^{t-1}(p-1)).
\]
This shows that $e_t = m_t$.
This proves the claim.

Taking $\Phi_n(x)$-torsion of \eqref{SES proof thm 1}, we obtain
\[
0 \longrightarrow \left( \frac{\Qp[x]}{\langle\Phi_{n}(x)\rangle}\right)^{s_n} \longrightarrow \left( \frac{\Qp[x]}{\langle\Phi_{n}(x)\rangle}\right)^{e_n} \longrightarrow \image(\eta_n)[\Phi_n(x)] \longrightarrow 0.
\]
The final term of the above short exact sequence was studied in Lemma~\ref{lemma 3 note 1}.
In particular, we obtained that
\[
\image(\eta_n)[\Phi_n(x)] \simeq \left( \bigoplus_{\substack{\val\in \A\\
n-\text{positive}}} W_{\val}\right) \otimes \frac{\Qp[x]}{\langle \Phi_n(x)\rangle}.
\]
By comparing $\Qp$-dimensions of the three terms in the above short exact sequence, we conclude that
\[
s_n + \theta_n= e_n
\]
and the result is now immediate.
\end{proof}

\begin{remark}
We can check that this reduces to \cite[Theorem~C and Corollary~4.3]{Lei23} when $K=\Q$.
\end{remark}

\begin{remark}
\label{rem: one div}
If \ref{star} is removed then $\image(\eta_n)$ could potentially be smaller than what is quantified by $\theta_n$.
In particular, we still have one divisibility
\[
 \prod_{n\geq 0} \langle (\Phi_n(x))^{e_n -\theta_n} \rangle \Big\vert \charac_{\Lambda}(\cM(\EC/K_{\cyc})^\vee) .
\]
\end{remark}
\subsection{Equivariant results}
Now, we study the structure of the fine Mordell--Weil group as a $\Zp\llbracket \cG \rrbracket$-module.
We remind the reader $\cG = \Gal(K_{\cyc}/\Q) \simeq \Gal(K/\Q) \times \Gal(K_{\cyc})/K = \sG \times \Gamma$.

In the previous section, we saw the structure of $T_p\cM(\EC/K_{(n)}) \otimes \Qp$ as a $\Qp[\Gamma_n]$-module.
On the other hand, we see from \eqref{1} and \eqref{2} that as $\Qp[\cG_n]$-module the following isomorphism is true:
\begin{equation}
\label{TpM tensor Qp equation}
T_p\cM(\EC/K_{(n)}) \otimes \Qp \simeq \bigoplus_{k=0}^n \bigoplus_{\substack{\val\in \A\\ k-\text{positive}}} \left( W_{\val}^{e_{\val,k}-1} \otimes \frac{\Qp[x]}{\langle \Phi_k(x)\rangle} \right).
\end{equation}
%But we know that as $\Zp[\cG_n]$-module, there is a pseudo-isomorphism; 
\edit{By \cite[Remark~2.4(b)]{Lei23}, we know that there is a pseudo-isomorphism of $\Lambda$-modules}
\[
\cM(\EC/K_{(n)})^\vee \sim \left( T_p \cM(\EC/K_{(n)})\right)^D.
\]
\edit{Viewing these modules as $\Zp[\cG_n]$-modules, we notice that the above homomorphism is also a $\Zp[\cG_n]$ map with finite kernel and cokernel.}
Using this (almost) identification, we prove the following result.

\begin{lemma}
For each positive integer $n$, there is a \edit{homomorphism of} $\Zp[\cG_n]$-module\edit{s with finite kernel and cokernel} 
\[
\psi_n : \cM(\EC/K_{(n)})^\vee \longrightarrow \bigoplus_{k=0}^n \bigoplus_{\substack{\val\in \A\\ k-\text{positive}}} \left( \overline{W}_{\val}^{e_{\val,k}-1} \otimes \frac{\Zp[x]}{\langle \Phi_k(x)\rangle} \right).
\]
\end{lemma}

\begin{proof}
First, consider an integral basis of the right side of \eqref{TpM tensor Qp equation}.
With this in hand, we know
\[
0 \longrightarrow \bigoplus_{k=0}^n \bigoplus_{\substack{\val\in \A\\ n-\text{positive}}} \left( \overline{W}_{\val}^{e_{\val,k}-1} \otimes \frac{\Zp[x]}{\langle \Phi_k(x)\rangle} \right) \longrightarrow T_p\cM(\EC/K_{(n)}) \longrightarrow \text{ finite } \longrightarrow 0.
\]
As $\Zp[G^{(i)}]$-modules, we have
\[
\left( \overline{V}_{i,j} \otimes \Zp \right)^D \simeq \overline{V}_{i,j} \otimes \Zp.
\]
In other words, $\overline{V}_{i,j} \otimes \Zp$ is self-dual.
It immediately follows that as $\Zp[\cG_n]$-modules, $\overline{W}_{\val} \otimes \frac{\Zp[x]}{\langle \Phi_k(x) \rangle}$ is self-dual for all $k$.
Taking $\Zp$-dual of the above short exact sequence yields
\[
0 \longrightarrow \left(T_p\cM(\EC/K_{(n)})\right)^D \longrightarrow \bigoplus_{k=0}^n \bigoplus_{\substack{\val\in \A\\ k-\text{positive}}} \left( \overline{W}_{\val}^{e_{\val,k}-1} \otimes \frac{\Zp[x]}{\langle \Phi_k(x)\rangle} \right) \longrightarrow \text{ finite } \longrightarrow 0.
\]
The claim follows immediately.
\end{proof}

The main result of this section is the following which should be considered as an equivariant refinement of Theorem~\ref{main result fine Selmer}.

\begin{theorem}
\label{thm 4.12}
Fix an odd prime $p$.
Let $K/\Q$ be an abelian extension which is disjoint from $\Q_{\cyc}$ and suppose that \ref{star} is satisfied for $\Gal(K/\Q) = \mathsf{G}$.
Let $\EC/\Q$ be an elliptic curve with good reduction at $p$ such that $\Zha(\EC/K_{(n)})[p^\infty]$ is finite for all $n$.
Then the following map is a pseudo-isomorphism of $\Lambda[\sG]$-modules
\[
\psi: \cM(\EC/K_{\cyc})^\vee \longrightarrow \bigoplus_{k=0}^\infty \bigoplus_{\substack{\val\in \A\\ k-\text{positive}}} \left( \overline{W}_{\val}^{e_{\val,k}-1} \otimes \frac{\Lambda}{\langle \Phi_k(x)\rangle} \right).
\]
\end{theorem}

\begin{proof}
For $n \gg 0$, there is a pseudo-isomorphism of $\Lambda[\sG]$-modules
\[
\psi' : \cM(\EC/K_{\cyc})^\vee \longrightarrow \cM(\EC/K_{(n)})^\vee.
\]
\edit{This follows from \cite[Theorem~3.3 and 3.6]{Lei23}.}
The result now follows by defining $\psi = \psi_n \circ \psi'$.
\end{proof}
\begin{remark}
\label{Rem non abelian}
In this section, the condition \ref{star} is used in an essential way in the proof of Lemma~\ref{lemma 3 note 1}.
In particular, we use the fact that irreducible representations of $\cG_n$ over $\Q$ stay irreducible after extending the scalars to $\Qp$.
There are many other \textit{non-abelian} groups which satisfy the same condition.
For instance, rational groups \edit{(by this we mean a finite group all of whose complex character values are rational)} whose Schur indices \edit{with respect to the field $\Q$} are \edit{all} 1 satisfy this property. \edit{This follows from the description of Schur indices given in Section~12.2 of \cite{Serre1977}.}
Some examples of such groups are symmetric groups $S_n$ and exceptional Weyl groups $E_6, E_7$ and $E_8$.
\edit{See for instance the main theorem of \cite{Benard}.}
One can check that when $G$ is such a group, all the results proved in this chapter can be generalized to this non-commutative setting using our methods.
In that case, one \edit{can} use the isomorphism \eqref{eq: isomorphism to answer ref} to \edit{determine} $e_{\val, k}$ \edit{uniquely} and then Definition~\ref{theta n} to define $\theta_n$.
Therefore, the structure of $\EC(K_{(n)})$ at each layer determines the structure of the fine Mordell--Weil group over the cyclotomic extensions.
\end{remark}

\begin{remark}\label{anticyclotomic remark}
Since the dihedral group $D_{2p^n}$ has its character values in $\Q(\zeta_p+\zeta_p^{-1})$ where $p$ does not split and satisfies the property of Schur indices mentioned in Remark~\ref{Rem non abelian} for each $n$, our results are also true for the anti-cyclotomic $\Zp$-extension of an imaginary quadratic $K_{\operatorname{ac}}/K$ whenever the Mordell--Weil rank growth is bounded.
In view of the growth number conjecture of Mazur \cite[\S~18, p.~201]{Maz84}, the Mordell--Weil rank of $\EC$ stays bounded along any $\Zp$-extension of the imaginary quadratic field $K$, unless the extension is the anti-cyclotomic one and the root number of $\EC/K$ is $-1$.
In some situations the conjecture is known, see for example \cite{Gre_PCMS, KMS, GHKL, KL25}.
\end{remark}

\section{Structure of the \texorpdfstring{$\pm$}{} Mordell--Weil group}
\label{sec: pm MW group}

Throughout this section, $\EC/\Q$ is an elliptic curve with good supersingular reduction at $p$ and for simplicity we assume that $a_p(\EC) =0$.

\subsection{Control Theorem and \edit{s}ome basic properties}
We first record that a control theorem is known for plus/minus groups by the work of B.D.~Kim \cite[\edit{Lemma~3.9}]{BDKim13} building on the ideas of \cite{Gre03}.
%and T.~Kitajima--R.~Otsuki \cite{KO18}.

\begin{theorem}[Control Theorem]
\label{Control Theorem ss}
Let $K$ be an abelian extension of $\Q$ and $p\edit{\geq}5$ be a fixed prime which is unramified in $K$.
Consider an elliptic curve $\EC/\Q$ with good supersingular reduction at $p$ and base-changed to $K$.
Then
\[
\psi_n: \Sel^{\pm}(\EC/K_{(n)}) \longrightarrow \Sel^{\pm}(\EC/K_{\cyc})^{\Gamma_n}
\]
is injective for all $n$ and the cokernel of $\psi_n$ has bounded growth as $n \to \infty$.
\end{theorem}

Let $L$ be an algebraic extension of $K$, for example $K_{(n)}$ or $K_{\cyc}$.
There exists a tautological short exact sequence
\begin{equation}
\label{eqn: tautological ss}
0 \longrightarrow \cM^{\pm}(\EC/L) \longrightarrow \Sel^{\pm}(\EC/L) \longrightarrow \Zha^{\pm}(\EC/L)[p^\infty] \longrightarrow 0.
\end{equation}

In view of the control theorem above (see Theorem~\ref{Control Theorem ss}), there exists a natural injective morphism
\[
m^{\pm}_n: \cM^{\pm}(\EC/K_{(n)}) \hookrightarrow \cM^{\pm}(\EC/K_{\cyc})^{\Gamma_n}.
\]
More information can be obtained about the map $m_n^{\pm}$ by imitating the proof of \cite[Theorem~5.5]{Lei23}.

\begin{theorem}
Let $K/\Q$ be an abelian extension and $p$ be a fixed odd prime which is unramified in $K$.
Consider an elliptic curve $\EC/\Q$ with good \emph{supersingular} reduction at $p$ with $a_p(\EC)=0$ and base-changed to $K$.
\begin{enumerate}
\item[\textup{(}i\textup{)}] There exists a pseudo-isomorphism
\begin{equation}
\label{pseudo for Mpm}
\cM^{\pm}(\EC/K_{\cyc}) \sim \bigoplus_{i=1}^{u^{\pm}} \frac{\Lambda}{\langle \Phi_{c^{\pm}_i}(x) \rangle}
\end{equation}
where $u^{\pm}$ are non-negative integers and $\Phi_{c^{\pm}_i}(x)$ are cyclotomic polynomials.
\item[\textup{(}ii\textup{)}] If $\Zha^{\pm}(\EC/K_{(n)})$ is finite then cokernel of $m_n^{\pm}$ is finite.
\end{enumerate}
\end{theorem}

We record the following corollary which will be useful for us.
This generalizes \cite[Remark~5.6]{Lei23}.

\begin{corollary}
\label{Remark 5.6 Antonio}   
With notation introduced previously and the assumption that $\Zha^{\pm}(\EC/K_{(n)})[p^\infty]$ is finite for each $n$,
\[
T_p\cM^{\pm}(\EC/K_{(n)}) = \bigoplus_{c_i^{\pm} \leq n} \frac{\Lambda}{\langle \Phi_{c^{\pm}_i}(x) \rangle}.
\]
\end{corollary}

\begin{remark}
When $K=\Q$, $p=2,3$ and $p\mid a_p$, F.~Sprung has defined the notion of $\flat$ and $\sharp$ Selmer groups; see \cite{Spr12}.
Many of the results in this section, can be extended to the case when $p\mid a_p$ but we refrain from doing so for the sake of clarity of arguments.
\end{remark}

\subsection{\texorpdfstring{$\Lambda$}{}-structure of the plus and minus Mordell--Weil groups}

In \cite{KP07} M.~Kurihara--R.~Pollack generalized the problem posed by Greenberg to elliptic curves over $\Q$ with supersingular reduction at $p$.
In \cite[Theorem~D]{Lei23}, Lei provided an affirmative answer to this question under the hypotheses that certain plus and minus Tate--Shafarevich groups are finite and that Kobayashi’s plus and minus main conjecture holds.
The plus/minus Iwasawa main conjecture is known for CM elliptic curves \cite{PK04} and for semi-stable elliptic curves \cite{BSTW}.
In this section we study the problem for elliptic curves over a number field $K$.
As will be evident to the reader, our result is strongly influenced by the work of Lei \cite{Lei23}.
However, we want to emphasize that our main result Theorem~\ref{thm 6.13} not only generalizes the work of Lei but also provides a refinement of \cite[Theorem~D]{Lei23} when $K=\Q$.

Throughout this section we assume that $K/\Q$ is an abelian extension such that $p$ is unramified in $K$.
Further we suppose suppose that \ref{star} is satisfied for $\Gal(K/\Q) = \mathsf{G}$.

\subsubsection{}
We first state a straight-forward result which is required for the remainder of the discussion.

\begin{lemma}
\label{Antonio 5.8}
Fix integers $n \geq 0$ and $k \geq 1$.
The natural map 
\begin{align*}
\EC(K_{(n)})/p^k &\longrightarrow \EC(K_{(n)}) \otimes \Qp/\Zp \\
\mathsf{P} \ (\bmod \ p^k) & \mapsto \mathsf{P} \otimes p^{-k}
\end{align*} 
is injective.
The image is given by $(\EC(K_{(n)}) \otimes \Qp/\Zp)[p^k]$.
It induces the following isomorphisms of $\Lambda$-modules:
\begin{align*}
\cM^{\pm}_{p^k}(\EC/K_{(n)}) &\simeq \cM^{\pm}(\EC/K_{(n)})[p^k]\\
\cM_{p^k}(\EC/K_{(n)}) &\simeq \cM(\EC/K_{(n)})[p^k]
\end{align*}
\end{lemma}

\begin{proof}
The proof in \cite[Proposition~5.8]{Lei23} goes through in this case.
\end{proof}

Set $\K'_n = K_{(n)}(\zeta_p) \otimes \Qp = K(\zeta_{p^{n+1}}) \otimes \Qp$ and $\mathcal{O}_{\K'_n}$ be the ring of integers of $\K'_n$, i.e., the integral closure of $\Zp$ in $\K'_n$.
Set $\fm_n$ to denote the Jacobson radical of $\mathcal{O}_{\K'_n}$.
Following \cite[Proposition~VII.6.3]{Sil86}, we know that there exists a short exact sequence
\[
0 \longrightarrow \hat{\EC}(\fm_n) \longrightarrow \EC(\K'_n) \longrightarrow \widetilde{\EC}(\mathcal{O}_{\K'_n}/\fm_n) \longrightarrow 0.
\]
Here $\hat{\EC}$ is the formal group law over $\Zp$ associated to a minimal Weierstrass model of $\EC$ and $\widetilde{\EC}$ denotes the$\pmod \edit{\fm_n}$ reduction of $\EC$.
Analogous to $\EC^{\pm}(\Kn)$, we may now define
\begin{align*}
\hat{\EC}^+(\fm_n) &:=
\left\{P\in \hat{\EC}(\fm_n) \mid \tr_{n/m+1} (P)\in \hat{\EC}(\fm_m), \text{ for }0\leq m < n\text{ and }m \text{ even }\right\},\\
\hat{\EC}^-(\fm_n) &:=
\left\{P\in \hat{\EC}(\fm_n) \mid \tr_{n/m+1} (P)\in \hat{\EC}(\fm_{m}),\text{ for }-1\leq m < n\text{ and }m 
\text{ odd }\right\}.
\end{align*}
In fact, $\hat{\EC}^{\pm}(\fm_n)$ is precisely the $p$-part of $\EC^{\pm}(\K'_n)$; see \cite[Definition~4.20]{Kat21}.
Define $\fm_{-1}$ to be the Jacobson radical of $\mathcal{O}_{\K_0}$ and $\hat{\EC}^{\pm}(\fm_{-1})$ be the $p$-part of $\EC(\K_0)$.

\begin{lemma}
\label{Lemma 1 note 20 May}
For $n,k \geq 0$,
\[
\cM^{+}_{p^k}(\EC/\Kn) \cap \cM^{-}_{p^k}(\EC/\Kn) = \EC(\K_0)/p^k.
\]
\end{lemma}

\begin{proof}
By \cite[Proposition~3.10]{Kat21} and the observation that $\hat{\EC}^{\pm}(\fm_n)$ is precisely the $p$-part of $\EC^{\pm}(\K'_n)$ we have
\begin{align*}
\hat{\EC}^{\pm}(\fm_n) \otimes \Qp/\Zp &\simeq \EC^{\pm}(\K'_n) \otimes \Qp/\Zp \text{ for all }n \geq 0\\
\hat{\EC}^{\pm}(\fm_{-1}) \otimes \Qp/\Zp &\simeq \EC(\K_0) \otimes \Qp/\Zp.
\end{align*}
Applying \cite[Proposition~4.21(2)]{Kat21} we see that
\[
0 \longrightarrow \EC(\K_0) \otimes \Qp/\Zp \longrightarrow \EC^{+}(\K'_n) \otimes \Qp/\Zp \oplus \EC^{-}(\K'_n) \otimes \Qp/\Zp \xrightarrow{\delta} \EC(\K'_n) \otimes \Qp/\Zp \longrightarrow 0,
\]
where the map $\delta$ is given by $(x,y) \mapsto x-y$.
Note that the first map is injective because $\EC(\K_{\cyc})[p^\infty]$ is trivial.
It follows from the definition of the map $\delta$ that
\[
\EC(\K_0) \otimes \Qp/\Zp = \EC^{+}(\K'_{\cyc}) \otimes \Qp/\Zp \cap \EC^{-}(\K'_{\cyc}) \otimes \Qp/\Zp.
\]
On the other hand, the definition of $\cM^{\pm}(\EC/\K_{\cyc})$ implies that
\[
\EC(\K_0) \otimes \Qp/\Zp \subseteq \cM^{+}(\EC/\K_{\cyc}) \cap \cM^{-}(\EC/\K_{\cyc}) \subseteq \EC^{+}(\K'_{\cyc}) \otimes \Qp/\Zp \cap \EC^{-}(\K'_{\cyc}) \otimes \Qp/\Zp.
\]
It is now clear that 
\[
\EC(\K_0) \otimes \Qp/\Zp = \cM^{+}(\EC/\K_{\cyc}) \cap \cM^{-}(\EC/\K_{\cyc}).
\]
Taking $\Gamma_n$-invariants we obtain
\[
\EC(\K_0) \otimes \Qp/\Zp = \cM^{+}(\EC/\K_{(n)}) \cap \cM^{-}(\EC/\K_{(n)}).
\]
The result follows by taking $p^k$-torsion.
\end{proof}

The reason we have to impose the condition that $p$ is unramified in $K$ is because of the above proposition which uses the result of Kataoka who works in that setting.

\begin{lemma}
\label{Lemma 2 note 20 May}
With notation introduced previously.
\begin{enumerate}
\item[\textup{(}i\textup{)}] $\cM^{\pm}_{p^k}(\EC/K) = \EC(K)/p^k$.
\item[\textup{(}ii\textup{)}] $V_p(\cM^{+}(\EC/K_{(n)}) \cap \cM^{-}(\EC/K_{(n)}))[\Phi_n(x)] = V_p \cM(\EC/K_{(n)})[\Phi_n(x)]$, when $n>0$.
\end{enumerate}
\end{lemma}

\begin{proof}
This is an analogue of \cite[Lemma~6.3]{Lei23}.
\begin{enumerate}
\item[\textup{(}i\textup{)}]
By definition of $\EC^{\pm}(\K_{\cyc})$, we observe that
\[
\EC(\K_0) \otimes \Qp/\Zp \subseteq (\EC^{\pm}(\K_{\cyc}) \otimes \Qp/\Zp)^{\Gamma} = \cM^{\pm}(\EC/\K_0).
\]
Now, we can take $p^k$-torsion and use Lemma~\ref{Antonio 5.8} to conclude that $\EC(\K_0)/p^k \subseteq \cM^{\pm}_{p^k}(\EC/\K_0)$.

Next, we use the definition of the plus/minus Mordell--Weil group,
\begin{align*}
\cM^{\pm}_{p^k}(\EC/K) &:= \ker\left( \EC(K)/p^k \longrightarrow \frac{H^1(\K_0, \EC[p^k])}{\cM_{p^k}^{\pm}(\EC/\K_0)} \right) \\
& \supseteq \ker\left( \EC(K)/p^k \longrightarrow \frac{H^1(\K_0, \EC[p^k])}{\EC(\K_0)/p^k} \right) \\
&= \EC(K)/p^k.
\end{align*}
The reverse inclusion is automatic from the definition.
This completes the proof of the claim.

\item[\textup{(}ii\textup{)}] 
By definition of plus/minus Mordell--Weil groups
\begin{align*}
\cM^{+}_{p^k}(\EC/K_{(n)}) \cap \cM^{-}_{p^k}(\EC/K_{(n)})&:= \ker\left( \EC(K_{(n)})/p^k \longrightarrow \frac{H^1(\Kn, \EC[p^k])}{\cM_{p^k}^{+}(\EC/\Kn) \cap \cM_{p^k}^{-}(\EC/\Kn)} \right) \\
& = \ker\left( \EC(K_{(n)})/p^k \longrightarrow \frac{H^1(\Kn, \EC[p^k])}{\EC(\K_0)/p^k} \right) \text{ by Lemma}~\ref{Lemma 1 note 20 May} \\
&= \ker\left( \EC(K_{(n)})/p^k \longrightarrow \frac{\EC(\Kn)/p^k}{\EC(\K_0)/p^k} \right).
\end{align*}
Writing this as an exact sequence,
\[
0 \longrightarrow \cM^{+}_{p^k}(\EC/K_{(n)}) \cap \cM^{-}_{p^k}(\EC/K_{(n)}) \longrightarrow \EC(K_{(n)})/p^k \longrightarrow \frac{\EC(\Kn)/p^k}{\EC(\K_0)/p^k}
\]
and taking projective limit (over $k$) and tensoring with $\Qp$ yields
\[
0 \longrightarrow  V_p(\cM^{+}(\EC/K_{(n)}) \cap \cM^{-}(\EC/K_{(n)})) \longrightarrow \EC(K_{(n)})^{\bullet} \longrightarrow \frac{\EC(\Kn)^{\bullet}}{\EC(\K_0)^{\bullet}}.
\]
For $n\geq 1$, this induces the exact sequence
\[
0 \longrightarrow  V_p(\cM^{+}(\EC/K_{(n)}) \cap \cM^{-}(\EC/K_{(n)}))[\Phi_n(x)] \longrightarrow \EC(K_{(n)})^{\bullet}[\Phi_n(x)] \xrightarrow{\eta_n} \EC(\Kn)^{\bullet}[\Phi_n(x)],
\]
where the map $\eta_n$ is the one described in Lemma~\ref{lemma 3 note 1}.
On the other hand, if we reconsider \eqref{SES proof thm 1} and take $\Phi_n$-torsion, we obtain
\[
0 \longrightarrow V_p\cM(\EC/K_{(n)})[\Phi_n(x)] \longrightarrow \EC(K_{(n)})^{\bullet}[\Phi_n(x)] \xrightarrow{\eta_n} \EC(\Kn)^{\bullet}[\Phi_n(x)].
\]
Comparing the kernel of $\eta_n$ gives the desired result. \qedhere
\end{enumerate}
\end{proof}

Before we continue the discussion, we want to introduce a new notation.
We set $r_n^{\pm}$ to be the exponent of $\frac{\Lambda}{\Phi_n(x)}$ appearing in the pseudo-isomorphism \eqref{pseudo for Mpm}.

\begin{proposition}
Keep the notation introduced previously and the assumption that $\Zha(\EC/K_{(n)})[p^\infty]$ and $\Zha^{\pm}(\EC/K_{(n)})[p^\infty]$ are finite for each $n$.
Then
\[
r_0^+ = r_0^- = e_0.
\]
\end{proposition}

\begin{proof}
Using Corollary~\ref{Remark 5.6 Antonio} we know that for $n\geq 0$
\[
V_p \cM^{\pm}(\EC/K_{(n)}) = \bigoplus_{j=0}^n \left( \frac{\Qp[x]}{\langle \Phi_j(x) \rangle}\right)^{r_j^{\pm}}.
\]
On the other hand, it follows from Lemma~\ref{Lemma 2 note 20 May}(i) that
\[
V_p\cM^{\pm}(\EC/K) = \EC(K)^{\bullet} = \Qp^{e_0}.
\]
The claim follows immediately.
\end{proof}

\subsubsection{}
The goal of this section is to study the sum of plus and minus Mordell--Weil groups.

\begin{lemma}
\label{bounded growth}
Fix $n \geq 0$.
The quotient $\frac{\EC(K_{(n)})/p^k}{\cM^{+}_{p^k}(\EC/K_{(n)}) + \cM^{-}_{p^k}(\EC/K_{(n)}) }$ has finite order bounded independently of $k$.
\end{lemma}

\begin{proof}
The proof is modelled on \cite[Lemma~6.5]{Lei23} but we provide the details for the convenience of the reader.
We begin with the observation that for $\mathsf{Q} \in \EC(\Kn)$, we have
\[
\tr_{n/m+1}(\mathsf{Q}) = \left(\prod_{i=m+2}^n \Phi_i(x)\right) \cdot \mathsf{Q}.
\]
Here we are using the identification of $\Lambda_n = \Zp[\cG_n]$ and $\frac{\Zp[\sG][x]}{\langle (1+x)^{p^n} - 1\rangle}$ obtained by sending $\gamma_n$ to $x+1$ where $\gamma_n$ is the image of $\gamma$ (the topological generator of $\Gamma$) in $\Lambda_n$.
Define
\[
\widetilde{\omega}_n^- = \prod_{\substack{i=1\\ i \text{ odd}}}^n \Phi_i(x) \text{ and } \widetilde{\omega}_n^+ = \prod_{\substack{i=0\\ i \text{ even}}}^n \Phi_i(x).
\]

\noindent \emph{Claim:} $\widetilde{\omega}_n^- \cdot \mathsf{Q} \in \EC^{+}(\K_n)$.

\noindent \emph{Justification:} By definition, it suffices to show that $\tr_{n/m+1}(\widetilde{\omega}_n^- \cdot \mathsf{Q}) \in \EC(\K_m)$ where $m$ is even.
We carry out the following computation when $m$ is even
\begin{align*}
    \tr_{n/m+1}(\widetilde{\omega}_n^- \cdot \mathsf{Q}) & = \left(\prod_{i=m+2}^n \Phi_i(x)\right) \left(\prod_{\substack{i=1\\ i \text{ odd}}}^n \Phi_i(x)\right)  \cdot\mathsf{Q} \\
    & = \left(\prod_{i=m+1}^n \Phi_i(x)\right) \left(\prod_{\substack{i=1\\ i\neq m+1\\i \text{ odd}}}^n \Phi_i(x)\right) \mathsf{Q}\\
    & = \tr_{n/m}(\mathsf{Q}') \in \EC(\K_m).
\end{align*}
This proves the claim.

In exactly the same way we can now show that $\widetilde{\omega}_n^+ \cdot \mathsf{Q} \in \EC^{-}(\K_n)$.
Since $\gcd(\widetilde{\omega}_n^+ , \widetilde{\omega}_n^-) = 1$, we know that there exist polynomials $A,B \in \Zp[x]$ and a positive integer $m$ such that
\[
A\widetilde{\omega}_n^+ + B \widetilde{\omega}_n^- = p^m.
\]
If $\mathsf{Q} \in \EC(K_{(n)})/p^k$ for some $k$ then we can further write
\[
\mathsf{P} : = p^m \mathsf{Q} = A\widetilde{\omega}_n^+ \mathsf{Q} + B \widetilde{\omega}_n^-\mathsf{Q} =: \mathsf{P}^- + \mathsf{P}^+.
\]
Now we take localization of $\mathsf{P}^{\pm}$ at $p$ and note that
\[
\operatorname{loc}_p(\mathsf{P}^{\pm}) \in \EC^{\pm}(\Kn) \subseteq \cM^{\pm}(\Kn).
\]
This forces that $\mathsf{P}^{\pm} \in \cM^{\pm}_{p^k}(K_{(n)})$.
The quotient is thus bounded by $p^m \rk_{\Z}(\EC(K_{(n)})$.
\end{proof}

\begin{proposition}
\label{prop 6.10}
With notation and assumptions introduced previously, 
\[
r_n^+ + r_n^- = 2e_n - \theta_n.
\]
\end{proposition}

\begin{proof}
This proof is adapted from \cite[Corollary~6.7]{Lei23}.
We consider the exact sequence
\[
0 \longrightarrow \cM^{+}_{p^k}(\EC/K_{(n)}) \cap \cM^{-}_{p^k}(\EC/K_{(n)}) \longrightarrow \cM^{+}_{p^k}(\EC/K_{(n)}) \oplus \cM^{-}_{p^k}(\EC/K_{(n)}) \longrightarrow \EC(K_{(n)})/p^k
\]
where the first map is given by diagonal embedding and the second map is given by $\mathsf{P} \oplus \mathsf{Q} \mapsto \mathsf{P} - \mathsf{Q}$.
We can take inverse limit and tensor with $\Qp$ for each of the objects appearing in the above exact sequence and combine it with Lemma~\ref{bounded growth} to obtain
\[
0 \rightarrow V_p(\cM^{+}(\EC/K_{(n)}) \cap \cM^{-}(\EC/K_{(n)}) )\rightarrow V_p(\cM^{+}(\EC/K_{(n)})) \oplus V_p(\cM^{-}(\EC/K_{(n)})) \rightarrow \EC(K_{(n)})^{\bullet} \rightarrow 0.
\]
Now using Theorem~\ref{main result fine Selmer}, Corollary~\ref{Remark 5.6 Antonio}, and Lemma~\ref{Lemma 2 note 20 May}(ii) we see that
\[
r_n^+ + r_n^- = s_n + e_n = (e_n - \theta_n) + e_n = 2e_n - \theta_n. \qedhere
\]
\end{proof}

In the next result, we give explicit formulae for $r_n^+$ and $r_n^-$ depending on the parity of $n$.

\begin{proposition}
\label{Prop 3 note 21 May}  
Let $n>0$.
With notation and assumptions introduced previously,
\[
r_n^+ = \begin{cases}
    e_n & \text{ if } n \text{ is even}\\
    e_n -\theta_n & \text{ if } n \text{ is odd}
\end{cases} \quad \quad \text{and} \quad \quad r_n^- = \begin{cases}
    e_n - \theta_n & \text{ if } n \text{ is even}\\
    e_n & \text{ if } n \text{ is odd.}
\end{cases}
\]
\end{proposition}

\begin{proof}
We prove the case when $n$ is even and the case of $n$ odd follows similarly.

It follows from the discussion in the proof of Lemma~\ref{bounded growth} that
\begin{equation}
\label{Prop 6.11 eqn 1}
\tilde{\omega}_n^{-} \cdot \EC(K_{(n)})^{\bullet} \subseteq V_p(\cM^{+}(\EC/K_{(n)})) \subseteq \EC(K_{(n)})^{\bullet}.
\end{equation}
But, since $n$ is even it follows from the definition of $\tilde{\omega}_n^-$ that $\gcd(\tilde{\omega}_n^-, \Phi_n(x))=1$.
We conclude that
\[
\tilde{\omega}_n^{-} \cdot \EC(K_{(n)})^{\bullet}[\Phi_n(x)] = \EC(K_{(n)})^{\bullet}[\Phi_n(x)].
\]
Combining this with the $[\Phi_n(x)]$-torsion part of \eqref{Prop 6.11 eqn 1}, we get
\[
V_p(\cM^{+}(\EC/K_{(n)}))[\Phi_n(x)] = \EC(K_{(n)})^{\bullet}[\Phi_n(x)]
\]
Thus
\[
\left( \frac{\Qp[x]}{\langle \Phi_n(x) \rangle}\right)^{r_n^+} \simeq \left( \frac{\Qp[x]}{\langle \Phi_n(x) \rangle}\right)^{e_n}.
\]
Thus, $r_n^+ = e_n$ and $r_n^- = e_n - \theta_n$ by Proposition~\ref{prop 6.10}.
This completes the proof.
\end{proof}

The following theorem (from the introduction) is now immediate.

\begin{theorem}
\label{thm 6.13}
Let $\EC/\Q$ be an elliptic curve and $p$ be a fixed prime of good \emph{supersingular} reduction.
Suppose that $a_p(\EC)=0$.
Let $K/\Q$ be an abelian extension such that $p$ is unramified in $K$ and suppose that \ref{star} is satisfied for $\Gal(K/\Q)$.
Suppose that $\Zha(\EC/K_{(n)})$, $\Zha^{\pm}(\EC/K_{(n)})$ are finite for all $n$.
\begin{align*}
    \charac_{\Lambda}(\cM^{+}(\EC/K_{\cyc})^\vee) & = \Bigg\langle x^{e_0} \prod_{\substack{n>0\\ n \text{ odd}}} \Phi_n(x)^{e_n - \theta_n} \prod_{\substack{n>0\\ n \text{ even}}} \Phi_n(x)^{e_n} \Bigg\rangle\\
     \charac_{\Lambda}(\cM^{-}(\EC/K_{\cyc})^\vee) & = \Bigg \langle x^{e_0} \prod_{\substack{n>0\\ n \text{ odd}}} \Phi_n(x)^{e_n} \prod_{\substack{n>0\\ n \text{ even}}} \Phi_n(x)^{e_n  - \theta_n} \Bigg\rangle.
\end{align*}
In particular,
\[
\gcd\left( \charac_{\Lambda}(\cM^{+}(\EC/K_{\cyc})^\vee), \charac_{\Lambda}(\cM^{-}(\EC/K_{\cyc})^\vee)\right) = \Bigg\langle x^{e_0} \prod_{n>0} \Phi_n(x)^{e_n - \theta_n} \Bigg\rangle.
\]
\end{theorem}

\begin{remark}
The authors have not seen an analogous result for the explicit description of $\charac_{\Lambda}(\cM^{\pm}(\EC/K_{\cyc})^\vee)$ written down in the literature when $K=\Q$.
\end{remark}

\begin{remark}
If we suppose that $K/\Q$ is abelian and $p$ is unramified in $K/\Q$ but do not assume that \ref{star} holds then for $s_n$ defined via \eqref{st} we have
\[
\charac_{\Lambda}(\cM^{+}(\EC/K_{\cyc})^\vee)  = \Bigg\langle x^{e_0} \prod_{\substack{n>0\\ n \text{ odd}}} \Phi_n(x)^{s_n} \prod_{\substack{n>0\\ n \text{ even}}} \Phi_n(x)^{e_n} \Bigg\rangle
\]
and similarly for $\charac_{\Lambda}(\cM^{-}(\EC/K_{\cyc})^\vee)$.
\end{remark}

As a final result in this section we intend to prove an analogue of Theorem~\ref{thm 4.12} for $\cM^{\pm}(\EC/K_{\cyc})^\vee$.

\begin{theorem}
\label{thm 5 note of June 16}
Fix an odd prime $p$.
Let $K/\Q$ be an abelian extension such that $p$ is unramified in $K$ and suppose that \ref{star} is satisfied for $\Gal(K/\Q) = \mathsf{G}$.
Let $\EC/\Q$ be an elliptic curve with good \emph{supersingular} reduction at $p$ such that $a_p(\EC)=0$ and $\Zha(\EC/K_{(n)})[p^\infty]$, $\Zha^{\pm}(\EC/K_{(n)})[p^\infty]$ are finite for all $n$.
Then the following map is a pseudo-isomorphism of $\Lambda[\sG]$-modules
\[
\psi: \cM^{\pm}(\EC/K_{\cyc})^\vee \longrightarrow \bigoplus_{k=0}^\infty \bigoplus_{\substack{\val\in \A\\ k-\text{positive}}} \left( \overline{W}_{\val}^{e_{\val,k}-t^{\pm}_k} \otimes \frac{\Lambda}{\langle \Phi_k(x)\rangle} \right).
\]
Here, 
\[
t^{+}_k = \begin{cases}
    0 & \text{ for } k \text{ even}\\
    1 & \text{ for } k \text{ odd}    
\end{cases} \ \text{ and } \  t^{-}_k = \begin{cases}
    0 & \text{ for } k \text{ odd or }0\\
    1 & \text{ for } k\edit{\neq 0} \text{ even}.
\end{cases}
\]
\end{theorem}

\begin{proof}
For $n \gg 0$, there is a pseudo-isomorphism of $\Lambda[\sG]$-modules 
\[
\psi' : \cM^{\pm}(\EC/K_{\cyc})^\vee \longrightarrow \cM^{\pm}(\EC/K_{(n)})^\vee.
\]
We prove the case for $\cM^{+}(\EC/K_{\cyc})^\vee$ and the other case is similar.

When \edit{$n$} is even, we note that $V_p(\cM^+(\EC/K_{(n)})) \subseteq \EC(K_{(n)})^{\bullet}$.
By Proposition~\ref{Prop 3 note 21 May} and upon comparing dimensions we get the following isomorphism of $\sG$-module\edit{s}
\[
V_p(\cM^+(\EC/K_{(n)}))[\Phi_k(x)] = \EC(K_{(n)})^{\bullet}[\Phi_k(x)] \simeq \bigoplus_{\substack{\val\in \A \\ k-\text{positive}}}\left( \overline{W}_{\val}^{e_{\val,k}} \otimes \frac{\Lambda}{\langle \Phi_k(x)\rangle}\right).
\]

On the other hand, for odd \edit{$n$}, we use the inclusion
\[
V_p(\cM(\EC/K_{(n)})) \subseteq V_p(\cM^+(\EC/K_{(n)})),
\]
compare $\Qp$-dimensions, and use Proposition~\ref{Prop 3 note 21 May} to conclude
\begin{align*}
V_p(\cM^+(\EC/K_{(n)}))[\Phi_k(x)] &= V_p(\cM(\EC/K_{(n)}))[\Phi_k(x)] \\
& \simeq \bigoplus_{\substack{\val\in \A\\ k-\text{positive}}} \left( W_{\val}^{e_{\val,k}-1} \otimes \frac{\Qp[x]}{\langle \Phi_k(x)\rangle} \right) \text{ by \eqref{TpM tensor Qp equation}}.
\end{align*}
Hence,
\[
V_p(\cM^+(\EC/K_{(n)})) = \bigoplus_{k=0}^n \bigoplus_{\substack{\val\in \A\\ k-\text{positive}}} \left( W_{\val}^{e_{\val,k}-t_k^+} \otimes \frac{\Qp[x]}{\langle \Phi_k(x)\rangle} \right).
\]
In view of this, we note that there is a pseudo-isomorphism \[
\bigoplus_{k=0}^n \bigoplus_{\substack{\val\in \A\\ k-\text{positive}}} \left( \overline{W}_{\val}^{e_{\val,k}-t^{+}_k} \otimes \frac{\Lambda}{\langle \Phi_k(x)\rangle} \right) \longrightarrow T_p(\cM^+(\EC/K_{(n)})).
\]
By self-duality as $\Lambda[\sG]$-modules we have
\[
\cM^+(\EC/K_{(n)})^\vee \sim T_p(\cM^+(\EC/K_{(n)}))^D.
\]
The result follows from this combined with the pseudo-isomorphism $\psi'$.
\end{proof}

\begin{remark}
\label{Rem on pm noncommutative and anticyclotomic}
One can also prove the non-commutative and anti-cyclotomic versions of the results proved in this section in the settings mentioned in Remark~\ref{Rem non abelian} and Remark~\ref{anticyclotomic remark} respectively.
\end{remark}

\section{Structure of the Selmer group}
\label{sec: Selmer group}

In this section, we study the structure of the dual Selmer group over the cyclotomic $\Zp$-extension \textit{when the base field is $\Q$}.

We begin by introducing some preliminary lemmas.

\begin{lemma}
\label{Lemma 1 Note 2}
Fix an odd prime $p$.
Let $K/\Q$ be an abelian extension which is unramified at $p$ and has Galois group $\sG$.
Set $\cG = \Gal(K_{\cyc}/\Q)$.
Let $\EC/\Q$ be an elliptic curve with good reduction at $p$.
\begin{enumerate}
\item[\textup{(}i\textup{)}]
Suppose that $p\nmid a_p(\EC)$, i.e., $p$ is a prime of good \emph{ordinary} reduction.
Then
\[
\left(\EC(\K_{\cyc}) \otimes \Qp/\Zp\right)^\vee \simeq \Zp\llbracket \cG\rrbracket.
\]
\item[\textup{(}ii\textup{)}]
Suppose that $a_p(\EC)=0$, i.e., $p$ is a prime of good \emph{supersingular} reduction.
Then 
\[
\left(\EC^{-}(\K_{\cyc})\otimes \Qp/\Zp\right)^\vee \simeq \Zp\llbracket \cG \rrbracket
\]
and the following short exact sequence holds
\[
0 \longrightarrow \left(\EC^{+}(\K_{\cyc})\otimes \Qp/\Zp\right)^\vee \longrightarrow \Zp\llbracket \cG \rrbracket \oplus \Zp[\sG] \xrightarrow{\Psi} \Zp[\sG] \longrightarrow 0,
\]
where $\Psi$ is given by the matrix $\begin{pmatrix}
-1 \\ \edit{\varphi + \varphi^{-1}}
\end{pmatrix}$ with $\varphi\in \sG$ the $p$-th power Frobenius map.
\end{enumerate}
\end{lemma}

\begin{proof}
\begin{enumerate}
\item[(i)] Set $\Delta = \Gal(K_{\cyc}(\zeta_p)/K_{\cyc}) \simeq \Gal(K(\zeta_p)/K) \simeq (\Z/p\Z)^\times$, where the first isomorphism crucially uses the fact that $\mu_p\cap K^\times$ is trivial.
It is clear that $K_{\cyc}(\zeta_p) = K(\zeta_{p^\infty})$.
Consider the cyclotomic character
\[
\chi: \Delta \longrightarrow \Zp^\times.
\]
For any \edit{$\Zp[\Delta]$}-module $M$, write $M^{\Delta=\chi^i}$ to denote the $\chi^i$-eigenspace of $M$ under the action of $\Delta$.
Set $\K'_{\edit{\cyc}} = K_{\edit{\cyc}}(\zeta_p) \otimes \Qp$.
Then with this notation
\[
\EC(\K'_{\cyc}) \simeq \bigoplus_{i=0}^{p-2} \EC(\K'_{\cyc})^{\Delta= \chi^i} \simeq \EC(\K_{\cyc}) \oplus \bigoplus_{i=1}^{p-2} \EC(\K'_{\cyc})^{\Delta= \chi^i}.
\]
We further observe that
\[
\left( \left( \EC(\K'_{\cyc}) \otimes \Qp/\Zp \right)^\vee\right)^{\Delta} \simeq \left(\EC(\K_{\cyc}) \otimes \Qp/\Zp \right)^\vee.
\]
Taking the $\Delta$-fixed part of \cite[Theorem~1.2(1)]{Kat21} we conclude that
\[
\left(\EC(\K_{\cyc}) \otimes \Qp/\Zp \right)^\vee \simeq \Zp\llbracket \cG \times \Delta \rrbracket^{\Delta} \simeq \Zp\llbracket \cG\rrbracket.
\]
\item[(ii)]
This claim is proven in a manner similar to (i).
Here, we take the $\Delta$-fixed part of the two sequences in \cite[Theorem~1.2(\edit{3})]{Kat21} to obtain the assertions.
\end{enumerate}
\end{proof}

\begin{lemma}
\label{Corollary 2 Note 2}
If $\dagger\in \{\emptyset, \pm\}$ then $\left(\EC^{\dagger}(\Q_{\cyc} )\otimes \Qp/\Zp\right)^\vee \simeq \Lambda = \Zp\llbracket \Gamma \rrbracket$.
\end{lemma}

We remind the reader that we write $\emptyset$ to mean $p\nmid a_p$ and $\EC^{\emptyset}(\K_{\cyc} )\otimes \Qp/\Zp = \EC(\K_{\cyc} )\otimes \Qp/\Zp$.

\begin{proof}
When $\dagger \in\{\emptyset, -\}$, the result is immediate from Lemma~\ref{Lemma 1 Note 2} by setting $K=\Q$ and $\sG = \{1\}$.

When $\dagger = +$, then $\left(\EC^{+}(\Q_{\cyc})\otimes \Qp/\Zp\right)^\vee = \ker(\Psi)$ where 
\begin{align*}
\Psi: \Lambda \oplus \Zp & \rightarrow \Zp\\
(f(\edit{x}), y) & \mapsto \edit{2}y - f(0)
\end{align*}
is the map in Lemma~\ref{Lemma 1 Note 2}(ii) with $K=\Q$.
It now follows that 
\begin{align*}
\Lambda & \xrightarrow{\sim }\ker(\Psi)\\
f(\edit{x}) & \mapsto \left( f(\edit{x}), \frac{f(0)}{\edit{2}}\right), 
\end{align*}
\edit{here we are using the implicit identification $\Lambda \simeq \Zp\llbracket x\rrbracket$.}
\end{proof}

\subsection{Case of good ordinary reduction}
In this section, we assume that $p$ is an odd prime of good ordinary reduction for $\EC/\Q$ and that $\Sha(\EC/\Q_{(n)})[p^{\infty}]$ is finite for all $n$.
By \cite[Corollary~3.0.6]{Lee20} we know that if the dual Selmer group over $\Q_{\cyc}$ has the following structure
\begin{equation}
\label{Lee expression}
\Sel(\EC/\Q_{\cyc})^\vee \sim \left( \bigoplus_{i=1}^r \frac{\Lambda}{\langle g_i^{d_i}(x)\rangle}\right) \oplus \left(\bigoplus_{\substack{j=1\\ f_j\geq 2}}^s \frac{\Lambda}{\langle \Phi_{a_j}^{f_j}(x)\rangle}\right) \oplus \left(\bigoplus_{k=1}^t \frac{\Lambda}{\langle \Phi_{b_k}(x)\rangle} \right) 
\end{equation}
where the irreducible polynomials $g_i(x)$'s are coprime to $\Phi_n(x)$ for all $n$, then
\begin{equation}
\label{Lee sha}
\Sha(\EC/\Q_{\cyc})[p^\infty]^\vee \sim \left( \bigoplus_{i=1}^r \frac{\Lambda}{\langle g_i^{d_i}(x)\rangle}\right) \oplus \left(\bigoplus_{\substack{j=1\\ f_j\geq 2}}^s \frac{\Lambda}{\langle \Phi_{a_j}^{f_j -1}(x)\rangle}\right).
\end{equation}

Our main goal in this section is to provide a refinement of the structure.

\begin{theorem}
\label{aj's are unique}
Fix an elliptic curve $\EC/\Q$ and prime $p>2$ of good \emph{ordinary} reduction.
Suppose that $\Sha(\EC/\Q_{(n)})[p^{\infty}]$ is finite for all $n$ and $\Zha(\EC/\Q_{\cyc})[p^\infty]$ is finite.
Then $g_i(x)$'s and $a_j$'s are distinct.
\end{theorem}

We remind the reader that $\Zha(\EC/\Q_{\cyc})[p^\infty]$ is expected to be finite; see \cite[Question~8.3]{Wut07}.

\begin{proof}
For a fixed $i$, suppose that there are $\alpha'$ terms with powers of $g_i(x)$ in (\ref{Lee expression}).
Recall from \eqref{alternative defn of FSG} that we have a right exact sequence
\[
\left( \EC(\Q_{\cyc,p}) \otimes \Qp/\Zp \right)^{\vee} \xrightarrow{\phi\edit{\otimes \operatorname{id}}} \Sel(\EC/\Q_{\cyc})^\vee \longrightarrow \Sel_0(\EC/\Q_{\cyc})^\vee \longrightarrow 0.
\]
In view of Lemma~\ref{Corollary 2 Note 2} we know that the above exact sequence becomes
\[
\Lambda \xrightarrow{\phi} \Sel(\EC/\Q_{\cyc})^\vee \longrightarrow \Sel_0(\EC/\Q_{\cyc})^\vee \longrightarrow 0.
\]
Upon tensoring the above sequence by $\Lambda/(g_i(x))$ we obtain the following exact sequence.
\[
\Lambda/\langle g_i(x)\rangle \xrightarrow{\phi \otimes \Lambda/\langle g_i(x)\rangle} \Sel(\EC/\Q_{\cyc})^\vee  \otimes \Lambda/\langle g_i(x)\rangle\longrightarrow \Sel_0(\EC/\Q_{\cyc})^\vee   \otimes \Lambda/\langle g_i(x)\rangle\longrightarrow 0.
\]
By the assumption that $\Zha(\EC/\Q_{\cyc})[p^\infty]$ is finite and using \cite[Theorem~C]{Lei23} we have that the last term of the sequence above is finite.
Therefore, by considering the characteristic ideals of the first two terms, we have,
\[
g_i(x)^{\alpha'}|g_i(x).
\]
Hence, we must have $\alpha'=1$ as desired.

To complete the proof of this theorem we need to show that each $\Phi_{a_j}$ term appears exactly once in the middle summand of \eqref{Lee expression}.

Suppose that $m\in \Z_{\geq 0}$ and $a_j=m$ for multiple $j$'s.
More precisely,
\[
\left(\bigoplus_{\substack{j=1\\ f_j\geq 2}}^c \frac{\Lambda}{\langle \Phi_{m}^{f_j}(x) \rangle}\right) \oplus \left(\bigoplus_{k=1}^d \frac{\Lambda}{\langle \Phi_{m}(x) \rangle}\right) \text{ with }c\geq 2
\]
appears in the structure of $\Sel(\EC/\Q_{\cyc})^{\vee}$ where it is possible that the second summand may vanish.
The goal is to show that $c\leq 1$.

Without loss of generality we may suppose that 
\[
f_c \geq f_{c-1} \geq \ldots \geq f_1 \geq 2.
\]
\emph{For ease of notation, we write $\omega = \langle \Phi_{m}(x) \rangle$ for the remainder of the proof.
}Suppose that there is a pseudo-isomorphism
\[
\image(\phi) \otimes \frac{\Lambda}{\omega} \sim \left( \frac{\Lambda}{\omega}\right)^{\alpha}.
\]
This implies that there exists a map
\[
\Lambda \otimes \frac{\Lambda}{\omega} \xrightarrow{\phi \otimes \operatorname{id}} \left( \frac{\Lambda}{\omega}\right)^{\alpha}
\]
with a finite cokernel.
This forces that $\alpha =0,1$.

Clearly, we have an injection $\image(\phi) \hookrightarrow \Sel(\EC/\Q_{\cyc})^\vee$ which induces the following map of elementary modules with a finite kernel
\[
\mathscr{E}(\image(\phi)) \longrightarrow \mathscr{E}(\Sel(\EC/\Q_{\cyc})^\vee).
\]
Therefore, we learn that the $\omega$-part of $\image(\phi)$ is of the shape $\left(\frac{\Lambda}{\omega^r}\right)^{\alpha}$ where $r \leq f_c$.

In view of the assumptions, we note that \cite[Theorem~C]{Lei23} and \cite[Corollary~3.0.6]{Lee20} are applicable and they together imply that the $\Lambda$-characteristic ideal of $\Sel_0(\EC/\Q_{\cyc})^\vee$ has $\omega^{c-1+d}$ as its $\omega$-part.
Since characteristic ideals are multiplicative in short sequences, we have
\[
\charac_{\Lambda}(\Sel(\EC/\Q_{\cyc})^\vee) = \charac_{\Lambda}(\image(\phi))\charac_{\Lambda}(\Sel_0(\EC/\Q_{\cyc})^\vee).
\]
%\noindent \emph{Claim:} $\alpha\neq 0,1$.
%\noindent \emph{Justification:}
\edit{
Comparing the $\omega$-part yields
\[
f_1 + \ldots + f_c + d = r\alpha + (c-1) + d.
\]
By definition, each $f_i > 1$ which forces that $\alpha \neq 0$. Hence $\alpha=1$ and
the above equality becomes
\[
f_1 + \ldots + f_c = r + (c-1) \leq f_c + c-1.
\]
Equivalently,
\[
f_1 + \ldots + f_{c-1} \leq c-1.
\]
This is not possible because each $f_i > 1$.
}

We have obtained a contradiction; hence, $c\leq 1$.
This completes the proof of the theorem.
\end{proof}

\begin{remark}
We have shown that the rank growth data of $\EC(\Q_{\cyc})$ along with information on the $\Lambda$-characteristic ideal of $\Sel(\EC/\Q_{\cyc})$ -- which can be obtained from an $L$-function by the Iwasawa Main Conjecture -- completely determines the structure of $\Sel(\EC/\Q_{\cyc})$.
\end{remark}

The next result is immediate upon combining Theorem~\ref{aj's are unique} and \eqref{Lee sha}.

\begin{corollary}
Fix an elliptic curve $\EC/\Q$ and an odd prime $p$ such that $p\nmid a_p$.
Suppose that $\Sha(\EC/\Q_{(n)})[p^{\infty}]$ is finite for all $n$ and that $\Zha(\EC/\Q_{\cyc})[p^\infty]$ is finite.
Then $\Sha(\EC/\Q_{\cyc})[p^\infty]^\vee$ is pseudo-isomorphic to a cyclic $\Lambda$-module.
\end{corollary}

\subsection{Case of good supersingular reduction}

Throughout this section, $\EC/\Q$ is an elliptic curve with good supersingular reduction at $p$ and for simplicity we assume that $a_p(\EC) =0$.
In particular, this condition is always satisfied if $p\geq 5$.

As in \cite{Lee20} we define the operator $\mathfrak{G}(X)$ for a $\Lambda$-module $X$ as follows
\[
\mathfrak{G}(X) = \varprojlim_{n} \left( \frac{X}{\omega_n X}[p^\infty]\right) \text{ where } \omega_n = (1+x)^{p^n} -1.
\]

\begin{theorem}
Let $\EC/\Q$ be an elliptic curve with good supersingular reduction at $p$.
Suppose that the supersingular Control Theorem holds for the tuple \textup{(}$\EC$, $p$, $K$\textup{)}.
Then
\[
\Zha^{\pm}
(\EC/K_{\cyc})^\vee \simeq \mathfrak{G}(\Sel^{\pm}(\EC/K_{\cyc})^\vee).
\]
\end{theorem}

\begin{proof}
Recall the map
\[
\psi_n: \Sel^{\pm}(\EC/K_{(n)}) \longrightarrow \Sel^{\pm}(\EC/K_{\cyc})^{\Gamma_n}
\]
appearing in the statement of the Control Theorem; see Theorem~\ref{Control Theorem ss}.
By definition, note that
\[
\varinjlim_n \ker(\psi_n) = \varinjlim_n \cok(\psi_n) = \{0\}
\]
and so are $\varprojlim_n \ker(\psi_n)^\vee[p^\infty]$ and $\varprojlim_n \cok(\psi_n)^\vee[p^\infty]$.
Now, using \cite[Lemma~2.1.4(3)]{Lee20} we have the following exact sequence
\begin{equation}
\label{will take limits}
0 \longrightarrow \cok(\psi_n)^\vee [p^\infty] \longrightarrow \frac{\Sel^{\pm}(\EC/K_{\cyc})^\vee}{\omega_n \Sel^{\pm}(\EC/K_{\cyc})^\vee}[p^\infty] \longrightarrow \Sel^{\pm}(\EC/K_{(n)})^\vee [p^\infty] \longrightarrow \ker(\psi_n)^\vee [p^\infty].
\end{equation}
Recall that
\[
\cM^{\pm}(\EC/K_{(n)})= \ker\left(\EC(K_{(n)}) \otimes \Qp/\Zp \longrightarrow \frac{H^1(\Kn, \EC[p^\infty])}{\cM^{\pm}(\EC/\K_{\cyc})} \right).
\]
Further note that $\EC(K_{(n)}) \otimes \Qp/\Zp$ is $p$-divisible and so is $\cM^{\pm}(\EC/\K_{\cyc}) = \EC(\K_{\cyc}) \otimes \Qp/\Zp$.
It now follows that $\cM^{\pm}(\EC/K_{(n)})$ is $p$-divisible and hence $\cM^{\pm}(\EC/K_{(n)})^\vee$ is $\Zp$-free.
In view of the tautological exact sequence \eqref{eqn: tautological ss}, we know that
\[
\Sel^{\pm}(\EC/K_{(n)})^\vee[p^\infty] = \Zha^{\pm}(\EC/K_{(n)})^\vee.
\]
The result follows upon taking projective limit of \eqref{will take limits}.
\end{proof}

\edit{In what follows, we suppose that $K/\Q$ is an abelian extension such that $p$ is unramified in $K$.}
Since $\Sel^{\pm}(\EC/K_{\cyc})^\vee$ is a finitely generated and torsion $\Lambda$-module \edit{\cite[Proposition~2.3 and Equation~(5.7)]{Kat21}}, we have the following result.

\begin{theorem}
\label{our analogue of Lee 3.0.6}
Suppose that $\Zha^{\pm}(\EC/K_{(n)})[p^\infty]$ is finite for all $n\geq 0$.
Suppose that
\[
\Sel^{\pm}(\EC/K_{\cyc})^\vee \sim \left( \bigoplus_{i=1}^d \frac{\Lambda}{\langle g^{l_i}_i(x) \rangle}\right) \oplus \left( \bigoplus_{\substack{m=1 \\ f_m\geq 2}}^f \frac{\Lambda}{\langle \Phi^{f_m}_{a_m}(x)\rangle}\right) \oplus \left( \bigoplus_{r=1}^t \frac{\Lambda}{\langle \Phi_{b_r}(x)\rangle}\right).
\]
Then, the following assertions hold
\begin{align*}
\Zha^{\pm}(\EC/K_{\cyc})[p^\infty]^\vee & \sim \left( \bigoplus_{i=1}^d \frac{\Lambda}{\langle g^{l_i}_i(x) \rangle}\right) \oplus \left( \bigoplus_{\substack{m=1 \\ f_m\geq 2}}^f \frac{\Lambda}{\langle \Phi^{f_m - 1}_{a_m}(x)\rangle}\right)\\
\cM^{\pm}(\EC/K_{\cyc})^\vee & \sim \left( \bigoplus_{m=1}^f \frac{\Lambda}{\langle \Phi_{a_m}(x)\rangle}\right) \oplus \left( \bigoplus_{r=1}^t \frac{\Lambda}{\langle \Phi_{b_r}(x)\rangle}\right),
\end{align*}
where the irreducible polynomials $g_i(x)$ are coprime to $\Phi_m(x)$ for all $m$.
\end{theorem}

\begin{proof}
The arguments in \cite[Corollary~3.0.6]{Lee20} go through in this case.
\end{proof}

We will now prove a supersingular analogue of Theorem~\ref{aj's are unique} where we give more insight into the precise structure of the $\Lambda$-modules when the elliptic curve is defined over $\Q$.

\begin{theorem}
\label{Thm 3 in note from June 16}
Fix an elliptic curve $\EC/\Q$ and an odd prime $p$ such that $a_p(\EC) =0$.
Suppose that $\Zha^{\pm}(\EC/\Q_{(n)})[p^{\infty}]$ is finite for all $n$ and that $\Zha(\EC/\Q_{\cyc})[p^\infty]$ is finite.
Then $g_i(x)$'s and $a_m$'s arising in Theorem~\ref{our analogue of Lee 3.0.6} are distinct.
Moreover, in the explicit structure of the dual of plus (resp.~minus) Selmer group, the integer $a_m$ is always even (resp.~odd or 0).
\end{theorem}

\begin{proof}
We will prove the result for $\Sel^+(\EC/K_{\cyc})^\vee$ and the proof of the other case is similar.
Moreover, the proof of the claim that $g_i(x)$'s are distinct is same as that of Theorem~\ref{aj's are unique}.
Now, let us prove the statement about $a_j$'s.
We begin by analyzing the $\Phi_{n}$-part of $\Sel^+(\EC/K_{\cyc})^\vee$.

Let $\omega = \Phi_n$-part of $\Sel^+(\EC/K_{\cyc})^\vee$ be
\[
\left( \bigoplus_{j=1}^c \frac{\Lambda}{\omega^{f_j}}\right) \oplus \left( \bigoplus_{k=1}^d \frac{\Lambda}{\omega}\right) \text{ where } f_c \geq f_{c-1} \geq \ldots \geq f_1 \geq 2.
\]
Observe that we have the following exact sequence
\begin{equation}
\label{intro theta}
(\EC^+(\Q_{\cyc,p}) \otimes \Qp/\Zp)^\vee \xrightarrow{\phi} \Sel^+(\EC/\Q_{\cyc})^\vee \longrightarrow \Sel_0(\EC/\Q_{\cyc})^\vee \longrightarrow 0.
\end{equation}
In view of Lemma~\ref{Corollary 2 Note 2} we know that the first module is isomorphic to $\Lambda$.
Suppose that
\[
\image(\phi) \otimes \frac{\Lambda}{\omega} \sim \left( \frac{\Lambda}{\omega}\right)^\alpha
\]
is a pseudo-isomorphism then $\alpha$ is either 0 or 1 (by the same argument as in Theorem~\ref{aj's are unique}).
As before, we have the following map of elementary modules with a finite kernel 
\[
\delta: \mathscr{E}(\image(\phi)) \longrightarrow \mathscr{E}(\Sel^{+}(\EC/\Q_{\cyc})^\vee).
\]
Note that the $\omega$-part of $\image(\phi)$ is $\left(\frac{\Lambda}{\omega^r}\right)^\alpha$ where $r \leq f_c$.

\smallskip

\underline{\emph{Case 1: When $n\geq 0$ and even.}}

%In this case suppose that $c\geq 2$.
In view of the assumptions, note that \cite[Theorem~C]{Lei23}, Proposition~\ref{Prop 3 note 21 May}, and Theorem~\ref{our analogue of Lee 3.0.6} imply that the $\Lambda$-characteristic ideal of $\Sel_0(\EC/\Q_{\cyc})^\vee$ has $\omega^{c-1+d}$ as its $\omega$-part.
The rest of the proof can be completed in the exact same way as in Theorem~\ref{aj's are unique} to show that $c \leq 1$ in this case.

\smallskip

\underline{\emph{Case 2: When $n\geq 1$ is odd.}}

In view of Proposition~\ref{Prop 3 note 21 May}, Theorem~\ref{our analogue of Lee 3.0.6}, and \cite[Theorem~C]{Lei23}, the $\omega$-part of $\Sel_{0}(\EC/\Q_{\cyc})^\vee$  is $\left(\frac{\Lambda}{\omega}\right)^{c+d}$.
Now, suppose that $c\geq 1$.
As in the proof of the previous case, we have
\[
f_1 + f_2 + \ldots + f_c + d = r\alpha + c + d.
\]
Since each $f_i > 1$ and $c\geq 1$, it forces that $\alpha =1$ and that
\[
f_1 + f_2 + \ldots + f_c = r + c \leq f_c + c.
\]
The only way for this to be satisfied is if $c=2$, $f_1 = 2$, and $f_2 =r$.
Therefore, the map $\delta$ of elementary modules is the following map with finite kernel
\[
\delta: \left(\frac{\Lambda}{\omega^r}\right) \longrightarrow \left(\frac{\Lambda}{\omega^2}\right) \oplus \left(\frac{\Lambda}{\omega^r}\right) \oplus \left(\frac{\Lambda}{\omega}\right)^d.
\]
Suppose that $\delta(\bar{1}) = \omega \cdot y$ for some some $y$ on the right side of the above map.
Then $\left( \frac{\omega^{r-1} \Lambda}{\omega^r} \right) \subseteq \ker(\delta)$ which is not possible.
This forces the image of $\delta \otimes \operatorname{id}$ to be infinite.
In other words, the image of
\[
\image(\phi) \otimes \frac{\Lambda}{\omega} \longrightarrow \Sel^+(\EC/\Q_{\cyc})^\vee \otimes \frac{\Lambda}{\omega}
\]
is also infinite.
However, we know from Theorem~\ref{our analogue of Lee 3.0.6} that the following natural map
\[
\Sel^{\pm}(\EC/\Q_{\cyc})^\vee \otimes \frac{\Lambda}{\omega} \longrightarrow \Sel_0(\EC/\Q_{\cyc})^\vee \otimes \frac{\Lambda}{\omega}
\]
is a pseudo-isomorphism.
This is clearly a contradiction in view of the exact sequence \eqref{intro theta}.

Thus, $c=0$ in this case.
This means that the term $\frac{\Lambda}{\langle \Phi^{f_n}_n(x) \rangle}$ does not appear in the decomposition of $\Sel^+(\EC/\Q_{\cyc})^\vee$ in this case.
\end{proof}

\begin{corollary}
Fix an elliptic curve $\EC/\Q$ and an odd prime $p$ such that $a_p =0$.
Suppose that $\Zha^{\pm}(\EC/\Q_{(n)})[p^{\infty}]$ is finite for all $n$ and that $\Zha(\EC/\Q_{\cyc})[p^\infty]$ is finite.
Then for $t\geq e_0$
\[
\text{cyclotomic part of } \gcd\left( \charac_{\Lambda}(\Sel^{+}(\EC/\Q_{\cyc})^\vee), \charac_{\Lambda}(\Sel^{-}(\EC/\Q_{\cyc})^\vee)\right) = \langle x^{t} \prod_{\substack{n\geq 1\\e_n>1}} \Phi^{e_n - 1}_n(x) \rangle.
\]
If $\Zha^{\pm}(\EC/\Q_{\cyc})[p^\infty]^{\Gamma}$ is finite then $t=e_0$.
\end{corollary}

\begin{proof}
By Theorem~\ref{Thm 3 in note from June 16} we know that the $\Phi_n$-part of the left hand side (of the equality in the statement of the corollary) and that of $\charac_{\Lambda}(\cM(\EC/\Q_{\cyc})^\vee)$ are the same when $n\geq 1$.

When $n=0$, the above Theorem~\ref{Thm 3 in note from June 16} only gives a divisibility.
The finiteness assumption on $\Zha^{\pm}(\EC/\Q_{\cyc})[p^\infty]^{\Gamma}$ ensures that the $\Phi_0$-part of $\Sel^{\pm}(\EC/\Q_{\cyc})^\vee$ and $\cM^{\pm}(\EC/\Q_{\cyc})^\vee$ are the same.
The result follows from Theorem~\ref{thm 6.13}.
\end{proof}

\begin{corollary}
Fix an elliptic curve $\EC/\Q$ and an odd prime $p$ of good \emph{supersingular} reduction such that $a_p(\EC)=0$.
Suppose that $\Zha^{\pm}(\EC/\Q_{(n)})[p^{\infty}]$ is finite for all $n$ and that $\Zha(\EC/\Q_{\cyc})[p^\infty]$ is finite.
Then $\Zha^{\pm}(\EC/\Q_{\cyc})[p^\infty]^\vee$ is pseudo-isomorphic to a cyclic $\Lambda$-module.
%Then the cyclotomic-part of $\Zha^{\pm}(\EC/\Q_{\cyc})[p^\infty]^\vee$ is cyclic as a $\Lambda$-module.
 Moreover, the $\Phi_m(x)$ appearing in the structure of $\Zha^{+}(\EC/\Q_{\cyc})[p^\infty]^\vee$ have $m$ even whereas, the $\Phi_m(x)$ appearing in the structure of $\Zha^{-}(\EC/\Q_{\cyc})[p^\infty]^\vee$ have $m$ odd or 0.
\end{corollary}

\bibliographystyle{amsalpha}
\bibliography{references}

\end{document}